\documentclass[a4paper,10pt]{article}
\usepackage{amsmath}
\usepackage{amssymb}
\usepackage{amsthm}

\usepackage{todonotes}

\usepackage{mathrsfs}
\usepackage{graphicx,subfigure}
\usepackage{paralist} %
\usepackage{enumerate}
\usepackage{hyperref,url}

\usepackage{bm} %
\usepackage{placeins}
\usepackage{color}
\usepackage[margin=30mm]{geometry}

\usepackage{verbatim}

\usepackage{authblk}

\newcommand{\dx}{\, \mathrm{dx}}
\newcommand{\ds}{\, \mathrm{ds}}

\newcommand{\be}{\begin{equation}}
\newcommand{\ee}{\end{equation}}
\newcommand{\been}{\begin{eqnarray*}}
\newcommand{\eeen}{\end{eqnarray*}}

\newcommand{\Omref}{\hat \Omega}

\newcommand{\hatXh}{\hat{\mathcal{U}}_h}

\theoremstyle{plain}
\newtheorem{thm}{Theorem}[section]

\newtheorem{lemma}[thm]{Lemma}

\theoremstyle{plain}
\newtheorem{remark}{Remark}[section]

\newtheorem{defn}{Definition}[section]

\newtheorem{alg}{Algorithm}[section]
\numberwithin{equation}{section}

\providecommand{\keywords}[1]{\textbf{\textit{Keywords:}} #1}
\providecommand{\subjclass}[1]{\textbf{\textit{MSC subject classification:}} #1}

\usepackage{adjustbox}
\usepackage[style=alphabetic,giveninits,citestyle=alphabetic,doi=false,isbn=false,url=false,eprint=false,maxnames=50,sorting=nyt]{biblatex}
\newcommand{\R}{\mathbb{R}}

\DeclareMathOperator{\Div}{div}

\DeclareMathOperator*{\argmin}{arg\,min}

\DeclareMathOperator{\id}{id}

\begin{document}

\title{Convergence of a steepest descent algorithm in shape optimisation using $W^{1,\infty}$ functions}
\author[1]{Klaus Deckelnick}
\affil[1]{Otto-von-Guericke-University Magdeburg, Department of Mathematics, Universit\"atsplatz 2, 39106 Magdeburg}
\author[2]{Philip J.~Herbert}
\affil[2]{Department of Mathematics, University of Sussex, Brighton, BN1 9RF, United Kingdom}
\author[3]{Michael Hinze\footnote{This work is part of the project P8 of the German Research Foundation Priority Programme 1962, whose support is gratefully acknowledged by the second and the third author.}}
\affil[3]{Mathematical Institute, University of Koblenz, Universit\"atsstr. 1, D-56070 Koblenz}
\date{\today}

\maketitle
\begin{abstract}
    Built upon previous work of the authors in \hyperlink{cite.0@DecHerHin21}{Deckelnick, Herbert, and Hinze, ESAIM: COCV 28 (2022)}, we present a general shape optimisation framework based on the method of mappings in the $W^{1,\infty}$ topology together with a suitable finite element discretisation. For the numerical solution of the respective discrete shape optimisation problems we propose a steepest descent minimisation algorithm with Armijo-Goldstein stepsize rule. We show that the sequence generated by this descent method globally converges, and under mild assumptions also, that every accumulation point of this sequence is a stationary point of the shape functional. Moreover, for the mesh discretisation parameter tending to zero we under mild assumptions prove convergence of the discrete stationary shapes in the Hausdorff complementary metric.
    To illustrate our approach we present a selection of numerical examples for PDE constrained shape optimisation problems, where we include numerical convergence studies which support our analytical findings.
\end{abstract}

\noindent\keywords{PDE constrained shape optimisation, $W^{1,\infty}$-steepest-descent, global convergence, finite element discretisation}\\
\noindent\subjclass{ 35Q93, 49Q10, 49J20}
\\
\renewcommand{\thefootnote}{\arabic{footnote}}

\section{Introduction}\label{intro}

We are interested in the numerical approximation of PDE constrained shape optimisation.
Our prototype problem will be of the form
\begin{equation}\label{prob0}
    \min \mathcal{J}(\Omega):= \int_\Omega j(\cdot, u,\nabla u) \dx,\, \Omega \in \mathcal{S},
\end{equation}
where $j$ is a real-valued function whose properties will be specified in Section \ref{sec:Preliminaries} and $u$ weakly solves the Poisson problem
\[
-\Delta u=f \text{ in } \Omega, \quad u=0 \text{ on } \partial \Omega.
\]
Furthermore, $\mathcal{S}$ is a collection of admissible domains contained in a given hold-all domain $D \subset \mathbb R^d$. We use the method of mappings and assume that each $\Omega \in \mathcal S$ is
represented by a bi--Lipschitz mapping $\Phi\colon D \rightarrow D$ as $\Omega=\Phi(\Omref)$, where $\Omref \Subset D$ is a fixed reference domain.
For domain variations one seeks a mapping $V^* \in W^{1,\infty}_0(D,\mathbb R^d)$ which forms a descent direction for the shape derivative, i.e. satisfies $\mathcal{J}'(\Omega)[V^*]<0$. The new domain is then obtained as $\Omega_{\rm new} = ({\rm id} + \alpha V^*)(\Omega)$ with $\alpha>0$ chosen suitably to ensure that the map $ \id + \alpha V^*$ is bi-Lipschitz. A common approach to determine a descent direction is to work in a Hilbert space $H \hookrightarrow W^{1,\infty}(D,\mathbb R^d)$ 
and then to find $V^*$ as the corresponding Riesz representative of $\mathcal J'(\Omega)$. Depending on the space dimension this may require the use of Sobolev spaces $H^m(D,\mathbb R^d)$ with a larger $m \in \mathbb N$ making the discretisation
of this approach cumbersome. In this work we follow instead the concept introduced in \cite{DecHerHin21}, \cite{DecHerHin23} and suggest to work directly in the space $W^{1,\infty}_0(D,\mathbb R^d)$ choosing
\begin{equation}\label{eq:generalMinimisationForDirection}
    V^* \in \argmin \left\{  \mathcal{J}'(\Omega)[V] :  V\in  W^{1,\infty}_0(D,\mathbb R^d), | D V | \leq 1 \mbox{ a.e. in } D \right\}
\end{equation}
as descent direction for the shape minimisation problem.
In the above, by $|DV|$, we mean the spectral norm of the matrix $DV$.
In order to approximate (\ref{prob0}) based on this idea we introduce the functional
\begin{equation}\label{probh}
     \mathcal{J}_h(\Omega_h) := \int_{\Omega_h} j(\cdot,u_h,\nabla u_h) \dx ,\quad  \Omega_h \in \mathcal{S}_h,
\end{equation}
where $u_h$ denotes the piecewise linear and continuous finite element function solving the discrete Poisson problem  \eqref{discstate} and $\mathcal S_h$ is a suitable approximation of $\mathcal S$.  For the numerical solution of the  discrete shape optimisation problem we propose a steepest descent method with Armijo step size rule which is realised in the $W^{1,\infty}-$ topology as described above.
{In fact $\mathcal S_h$ is built upon piecewise linear and continuous approximations $\Phi_h$ of the mapping $\Phi$, which in turn are induced by piecewise linear and continuous vector fields $V_h^*$ solving the discrete counterpart of \eqref{eq:generalMinimisationForDirection}.
We here note that the use of piecewise linear and continuous finite elements is perfectly tailored to the numerical treatment of our approach, since they belong to $W^{1,\infty}$, and both problems \eqref{eq:generalMinimisationForDirection} and \eqref{probh} can be discretised on the same triangulation $\mathcal{T}_h$.}
It is the purpose of this paper to analyse the resulting numerical method both for a fixed mesh width $h$ and for the case that $h$ tends to zero thereby justifying the underlying approach.
The main contributions of this work are 
\begin{itemize}
    \item Theorem \ref{conv1}, where global convergence of the steepest descent method is shown for a fixed discretisation parameter, and under mild assumptions also, that every accumulation point of this sequence is a stationary point of the discrete shape functional;
    \item Theorem \ref{conv3}, where it is shown that under suitable conditions a sequence of discrete stationary shapes converges with respect to the Hausdorff complementary metric to a stationary point of (\ref{prob0}). 
\end{itemize}
An important ingredient in the proof of Theorem \ref{conv3} is the continuity of the Dirichlet problem with respect to the Hausdorff complementary metric which is usually expressed in terms of $\gamma$--convergence or (equivalently) Mosco--convergence. Our analysis
is inspired by the work \cite{CZ06} of Chenais and Zuazua, who obtain the convergence of a sequence of discrete minimal shapes, obtained by some finite element approximation, to a minimum of the continuous problem. In \cite{CZ06}, Mosco--convergence is a consequence of
the assumption that the complementary sets of the discrete optimal shapes have a uniformly bounded number of connected components. In contrast, in our setting it will be more convenient to work with a uniform capacity density condition, see Theorem \ref{compact0}.
A convergence result for a shape optimisation problem in the class of convex domains has recently been obtained by Bartels and Wachsmuth, \cite{BarWac20} under a condition that will  also appear in our work. \\
In special settings a priori estimates for finite element approximations of shape optimisation problems have been proved. Here we refer to the works of Kiniger and Vexler \cite{KV13} and  Fumagalli et al.~\cite{FPV15}, where graph settings are
considered, and of Eppler et al.~\cite{EHS07} for star-shaped domains. \\
Another aspect that has been examined from the viewpoint of numerical analysis is the approximation of the shape derivative. In \cite{HPS15} Hiptmair, Paganini, and Sargheini study the finite element approximation of the shape derivative under appropriate regularity assumptions of the state and the adjoint state.
In \cite{gong2021discrete} Gong and Zhu propose a finite element approximation to the boundary form of the shape derivative in PDE constrained shape optimisation.
Zhu and Gao in \cite{zhu2019convergence} numerically analyse a mixed finite element approximation of the shape gradient for Stokes flow, and Zhu,  Hu and Liao in \cite{zhu2020convergence} provide numerical analysis for the finite element approximation of shape derivatives in eigenvalue optimisation for the Poisson problem. For additional information on the subject of shape optimisation we refer the reader to the seminal works of Delfour and Zol\'esio \cite{DelZol11},  of Sokolowski and Zol\'esio \cite{SZ92}, and the recent overview article \cite{ADJ21} by Allaire, Dapogny, and Jouve, where also a comprehensive  bibliography on the topic can be found.

\paragraph{Outline:} In Section \ref{sec:Preliminaries} we provide preliminaries for the formulation and the numerical analysis of our PDE constrained shape optimisation problem.
In Section \ref{sec:convergenceOfAlgorithm} we prove global convergence for the steepest descent method applied to problem \eqref{probh}, and in Section \ref{sec:ConvergenceOfStationaryPoints} prove convergence of discrete stationary points to a stationary point for the limit problem \eqref{prob0}.
In Section \ref{sec:experiments} we provide numerical experiments which support our theoretical findings.

\section{Preliminaries}\label{sec:Preliminaries}
\subsection{Setting of the problem}

\noindent
Let $D \subset \mathbb R^d$ be an open, convex, polygonal hold-all domain and  $\Omref \Subset D$  a fixed reference domain. We define
\begin{displaymath}
\mathcal U:= \lbrace \Phi: \bar D \rightarrow \bar D \, | \, \Phi \mbox{ is a bilipschitz map}, \Phi=\mbox{id} \mbox{ on } \partial D \rbrace
\end{displaymath}
and our set of admissible shapes as
\begin{displaymath}
\mathcal S:= \lbrace \Omega \subset D \, | \, \Omega= \Phi(\Omref) \mbox{ for some } \Phi \in \mathcal U \rbrace.
\end{displaymath}
Let us  consider the shape optimisation problem
\begin{displaymath}
\min_{\Omega \in \mathcal S} \mathcal J(\Omega) = \int_{\Omega} j(x,u(x),\nabla u(x)) \dx,
\end{displaymath}
where $u \in H^1_0(\Omega)$ is the unique solution of 
\begin{equation} \label{state}
\int_{\Omega} \nabla u \cdot \nabla \eta \dx = \langle f,\eta \rangle \qquad \mbox{ for all } \eta \in H^1_0(\Omega).
\end{equation}
Our definition of $\mathcal S$ allows us to interpret $\mathcal U$ as the set of controls for a PDE--constrained optimisation problem. 
In what follows we assume that $f \in H^1(D)$ and that $j \in C^2(D \times \mathbb R \times \mathbb R^d)$ satisfies 
\begin{eqnarray}
| j(x,u,z) | + | j_x(x,u,z)| + | j_{xx}(x,u,z)|  & \leq  & \varphi_1(x)+ c_1 \bigl( | u|^q + |z|^2 \bigr); \label{jest} \\
| j_u(x,u,z) | + | j_{xu}(x,u,z) | & \leq  &  \varphi_2(x) + c_2 \bigl(  |u |^{q-1} + | z|^{2-\frac{2}{q}} \bigr);   \label{j1est} \\
| j_z(x,u,z) | + | j_{xz}(x,u,z) |  & \leq   & \varphi_3(x) + c_3 \bigl( |u|^{\frac{q}{2}} + |z| \bigr);  \label{j2est}   \\
| j_{uu}(x,u,z) | & \leq  &\varphi_4(x) + c_4 \bigl( |u|^{q-2} + |z|^{2- \frac{4}{q}} \bigr); \label{j3est} \\
  | j_{zz}(x,u,z) | & \leq & \varphi_5(x) \label{j4est}
\end{eqnarray}
for all $(x,u,z) \in  D \times  \mathbb R \times \mathbb R^d$. Here, $2 \leq q < \infty$ if $d=2$ and $q= \frac{2d}{d-2}$ if $d \geq 3$. Also, $\varphi_1,\ldots,\varphi_5$ are non-negative with 
$\varphi_1 \in L^1(D), \varphi_2 \in L^{\frac{q}{q-1}}(D), \varphi_3 \in L^2(D), \varphi_4 \in L^{\frac{q}{q-2}}(D)$ and 
$\varphi_5 \in L^\infty(D)$.
{Note that the choice of $q$ implies the continuous embedding $H^1_0(D) \hookrightarrow L^q(D)$, so that there exists $c>0$ with
\begin{equation} \label{Dembed}
\Vert v \Vert_{L^q} \leq c \Vert v \Vert_{H^1} \qquad \mbox{ for all } v \in H^1_0(D).
\end{equation} } 
\noindent
It is well known that the shape derivative of $\mathcal J$ is given by
\begin{eqnarray}
 \mathcal J'(\Omega)[V] 
 & = &  \int_{\Omega} \Bigl( j(\cdot,u,\nabla u) \Div V + j_x(\cdot,u,\nabla u)\cdot  V  - j_z(\cdot,u,\nabla u) \cdot DV^{\mathsf{T}} \nabla u \Bigr) \dx    \label{firstvar} \\
 & & + \int_{\Omega} \Bigl(  \bigl( DV + D V^{\mathsf{T}} - \Div V I \bigr)  \nabla u \cdot \nabla p
+ {\Div}( f V) p   \Bigr) \dx \nonumber
\end{eqnarray}
for all $V \in W^{1,\infty}_0(D, \mathbb R^d)$. Here,
$p \in H^1_0(\Omega)$ is the solution of the adjoint problem
\begin{equation} \label{adj}
\int_{\Omega} \nabla p \cdot \nabla \eta \dx = \int_{\Omega} \bigl( j_u(\cdot,u,\nabla u) \eta + j_z(\cdot,u,\nabla u) \cdot \nabla \eta \bigr) \dx \quad \mbox{ for all } \eta \in H^1_0(\Omega).
\end{equation}
{We observe that \eqref{jest}--\eqref{j2est} together with \eqref{Dembed} imply that the integrals on the right hand side of \eqref{firstvar} and \eqref{adj} exist.}
Finding a global minimiser of $\mathcal J$ is usually a very hard task so that numerical methods aim to approximate stationary points, i.e. sets $\Omega \in \mathcal S$ that
satisfy $\mathcal J'(\Omega)[V]=0$ for all $V \in W^{1,\infty}_0(D, \mathbb R^d)$. 

\subsection{Discretisation}
In order to define a corresponding numerical method we choose an admissible triangulation  $\mathcal { \hat T}_h$ of $\bar D$ and define
\begin{displaymath}
\hatXh := \lbrace \Phi_h \in C^0(\bar D,\mathbb R^d) \, | \, \Phi_{h | \hat T} \in P^1(\hat T, \mathbb R^d), \hat T \in \mathcal{ \hat T}_h, \Phi_h \mbox{ is injective},
\Phi_h= \mbox{id} \mbox{ on } \partial D \rbrace.
\end{displaymath}

\noindent
We start with the following observation.
\begin{lemma} \label{homeo}  Let $\Phi_h \in \hatXh$. Then $\Phi_h$ is a bilipschitz map from $\bar D$ onto $\bar D$.
\end{lemma}
\begin{proof} Denoting by $\mbox{deg}$ the Brouwer degree and using that $\Phi_h=\mbox{id}$ on $\partial D$, we have for every $p \in D$ that
\begin{displaymath}
\mbox{deg}(\Phi_h,D,p) = \mbox{deg}(\mbox{id},D,p)=1.
\end{displaymath}
Hence we deduce from the existence property of the degree that there exists $x \in D$ with $p=\Phi_h(x)$, and therefore $D \subset \Phi_h(D)$. Next we claim that $D$ is closed in $\Phi_h(D)$. To see this,
let $(p_n)_{n \in \mathbb N}$ be a sequence in $D$ such that $p_n \rightarrow p$ as $n \rightarrow \infty$ for some $p \in \Phi_h(D)$, say $p=\Phi_h(x)$ with
$x \in D$. If $p \in \partial D$, then $\Phi_h(x)=p=\Phi_h(p)$ and hence we obtain in view of  the injectivity of $\Phi_h$ that $x=p$,
a contradiction. Hence $p \in D$. As $D$ is also open in $\Phi_h(D)$ and $\Phi_h(D)$ is connected we infer that $D=\Phi_h(D)$.
Recalling again that $\Phi_h=\mbox{id}$ on $\partial D$ we see that $\Phi_h:\bar D \rightarrow \bar D$ is bijective. Finally, using the fact that $\Phi_h$ is piecewise linear and injective together with
the convexity of $D$ it is not difficult to show
that  there exists a constant $K >1$ depending on $\Phi_h$ such that
\begin{equation} \label{bilip}
\frac{1}{K} | x - y | \leq | \Phi_h(x) - \Phi_h(y) | \leq K | x - y | \qquad \forall x,y \in \bar D.
\end{equation}
\end{proof}

\noindent
Similarly as in \cite[Section 3.2]{BarWac20} we shall define our discrete admissible domains  via transformations of $\Omref$ from the set $\hatXh$. In what follows we
 assume that $\Omref$ is an open polygonal  domain such that
$\overline{\Omref} = \bigcup_{\hat T \in \mathcal{ \hat T}^{\tiny \mbox{ref}}_h} \hat T \subset D$, 
where  $\mathcal{ \hat T}^{\tiny \mbox{ref}}_h \subset \mathcal{ \hat T}_h$.
For later purposes we suppose in addition  that $\Omref$  satisfies the following exterior corkscrew condition:
\begin{equation} \label{corkscrew}
\exists \lambda \in (0,1) \, \exists s_0>0 \,  \forall \hat x \in \partial \Omref \,  \forall s \in (0,s_0) \, \exists \hat y \in B_s(\hat x): \quad B_{\lambda s}(\hat y) \subset B_s(\hat x) \cap \complement \Omref.
\end{equation}
\noindent
In the above $\complement \Omref$ denotes the complement of $\Omref$. 
We then define
\begin{equation} \label{admdisc}
\mathcal S_h:= \lbrace \Omega_h \subset D \, | \, \Omega_h= \Phi_h(\Omref) \mbox{ for some } \Phi_h  \in \hatXh \rbrace.
\end{equation}
Note that in view of Lemma \ref{homeo} sets $\Omega_h \in \mathcal S_h$ are triangulated 
in a natural way via $\mathcal T_{\Omega_h} = \lbrace \Phi_h(\hat T), \, \hat T \in \mathcal{ \hat T}^{\tiny \mbox{ref}}_h \rbrace$.
Given a triangulation of this form we introduce
\begin{displaymath}
X_{\Omega_h}:= \lbrace \eta_h \in C^0(\overline{\Omega_h}) \, | \, \eta_{h|T} \in P_1(T), T \in \mathcal T_{\Omega_h}, \,
\eta_h=0 \mbox{ on } \partial \Omega_h \rbrace. 
\end{displaymath}
Our discrete shape optimisation problem now reads:
\begin{equation} \label{discshapeopt}
\min \mathcal J_h(\Omega_h):= \int_{\Omega_h} j(x,u_h(x),\nabla u_h(x)) \dx,
\end{equation}
where $u_h \in X_{\Omega_h}$ is the unique solution of 
\begin{equation} \label{discstate}
 \qquad \int_{\Omega_h} \nabla u_h \cdot \nabla \eta_h \dx = \langle  f, \eta_h \rangle \qquad \mbox{ for all } \eta_h \in X_{\Omega_h}.
\end{equation}
We remark that we have chosen linear finite elements merely for convenience and that one may take any conforming finite element space in
order to approximate the solution of \eqref{state}. \\[2mm]
\noindent
Let us fix $\Omega_h=\Phi_h(\Omref) \in \mathcal S_h$ for some $\Phi_h \in \hatXh$.
In order to define a suitable perturbation of $\Omega_h$ we let
\begin{equation} \label{vhspace}
\mathcal V_{\Phi_h} := \lbrace V_h \in C^0(\bar D,\mathbb R^d) \, | \, V_{h|T} \in P_1(T,\mathbb R^d), T = \Phi_h(\hat T), \hat T \in \mathcal{ \hat T}_h, \, V_h=0 \mbox{ on }
\partial D \rbrace.
\end{equation}
Suppose that  $V_h \in \mathcal V_{\Phi_h}$ with $| D V_h | \leq 1$ in $\bar D$. Clearly, $\Phi_h + t V_h \circ \Phi_h$ belongs to $\hatXh$ provided that $| t | <1$.
Hence $\Omega_{h,t}:=(\mbox{id}+ t V_h)(\Omega_h) = (\Phi_h + t V_h \circ \Phi_h)(\Omref) \in \mathcal S_h$ 
if  $|t|<1$ and  we may define $\displaystyle \mathcal J_h'(\Omega_h)[V_h]:= \frac{d}{dt} \mathcal J_h(\Omega_{h,t})_{|t=0}$. The formula for $\mathcal J'(\Omega_h)[V_h]$ is obtained
analogously to the continuous case. As the corresponding arguments will appear in the proof of Lemma \ref{decrease} below we here  merely state its form:
\begin{eqnarray}
 \mathcal J_h'(\Omega_h)[V_h]  
& = & \int_{\Omega_h} \Bigl( j(\cdot,u_h,\nabla u_h) {\Div} V_h +  j_x(\cdot,u_h,\nabla u_h) \cdot V_h - j_z(\cdot,u_h,\nabla u_h) \cdot DV_h^{\mathsf{T}} \nabla u_h \Bigr) \dx  \nonumber  \\
& & + \int_{\Omega_h} \Bigl( 
\bigl( DV_h + D V_h^{\mathsf{T}} - {\Div} V_h I \bigr)  \nabla u_h \cdot \nabla p_h
+ {\Div}( f V_h) p_h   \Bigr) \dx, \label{discsd}
\end{eqnarray}
where $p_h \in X_{\Omega_h}$ solves
\begin{equation} \label{discadj}
\int_{\Omega_h} \nabla p_h \cdot \nabla \eta_h \dx = \int_{\Omega_h} \bigl( j_u(\cdot,u_h,\nabla u_h) \eta_h + j_z(\cdot,u_h,\nabla u_h) \cdot \nabla \eta_h \bigr)  \dx \quad \mbox{ for all } \eta_h \in X_{\Omega_h}.
\end{equation}

\subsection{Descent algorithm} 

\noindent
With the notation introduced in the previous section we may now formulate a steepest descent method with Armijo search:
\begin{alg}[Steepest descent]\label{alg:Steepest}
\ \\ [2mm]
0. Let $\Omega^0_h:= \Omref, \Phi^0_h= \id$. \\[2mm]
For k=0,1,2,\ldots: \\[2mm]
1. If $\mathcal J_h'(\Omega^k_h)=0$, then stop. \\[2mm]
2. Choose $V^k_h \in \mathcal V_{\Phi^k_h}$ such that 
\begin{displaymath}
V^k_h = \argmin \lbrace J_h'(\Omega^k_h)[W_h] \, | \, W_h \in \mathcal V_{\Phi^k_h}, \, | DW_h | \leq 1 \mbox{ in } \bar D \rbrace.
\end{displaymath}
3. Choose the maximum $t_k \in \lbrace  \frac{1}{2}, \frac{1}{4}, \ldots \rbrace$ such that 
\begin{displaymath}
\mathcal J_h \bigl( (\id + t_k V^k_h)(\Omega^k_h) \bigr) - \mathcal J_h(\Omega^k_h) \leq \gamma t_k \mathcal J_h'(\Omega^k_h)[V^k_h].
\end{displaymath}
4. Set $\Phi_h^{k+1}:= (\mbox{id} + t_k V_h^k )\circ \Phi_h^k, \, \Omega^{k+1}_h:= (\mbox{id}+ t_k V^k_h)(\Omega^k_h)$.
\end{alg}
Here, $\gamma \in (0,1)$ is a fixed constant. In view of the  remarks after \eqref{vhspace} the algorithm produces  a sequence $(\Phi_h^k)_{k \in \mathbb N_0} \subset
\hatXh$  such that $\Omega_h^k= \Phi_h^k(\Omref) \in \mathcal S_h, k \in \mathbb N_0$. Our aim is to show that
\begin{displaymath}
\Vert \mathcal J_h'(\Omega^k_h) \Vert: = \sup \lbrace \mathcal J_h'(\Omega^k_h)[W_h] \, | \, W_h \in \mathcal V_{\Phi^k_h}, | DW_h | \leq 1 \mbox{ in } \bar D \rbrace \rightarrow 0, \quad \mbox{ as }
k \rightarrow \infty.
\end{displaymath}

\section{Convergence of the descent algorithm}\label{sec:convergenceOfAlgorithm}
\noindent
{In the present section we investigate the global convergence of the descent Algorithm \eqref{alg:Steepest}, where the discretisation parameter $h$ is kept fixed.} As a first step we note the following a--priori bounds on the discrete state and its adjoint state.
 
 \begin{lemma}  Let $\Omega_h=\Phi_h(\Omref) \in \mathcal S_h$ and $u_h,p_h \in X_{\Omega_h}$ the solutions of \eqref{discstate}, \eqref{discadj}
respectively. Then
\begin{equation} \label{discapriori}
\Vert u_h \Vert_{H^1}  \leq c  \Vert f \Vert_{L^2}, \quad \Vert p_h \Vert_{H^1} \leq c \bigl( 1 + \Vert f \Vert_{L^2}^{q-1} \bigr),
\end{equation}
where the constant $c$ only depends on $d,j$ and $D$. Here we think of $u_h$ and $p_h$ as being extended by zero to $D$.
\end{lemma}
\begin{proof} The first estimate is standard. In order to prove the bound on $p_h$ we test \eqref{discadj} with $\eta_h=p_h \in X_{\Omega_h}$ and use \eqref{j1est}, \eqref{j2est},
H\"older's inequality and  {\eqref{Dembed}} to obtain
\begin{eqnarray*}
\lefteqn{  \int_D | \nabla p_h |^2 \dx = 
\int_{\Omega_h} | \nabla p_h |^2 \dx    } \\
 & \leq & \int_{\Omega_h} \left[ \bigl( \varphi_2 + c_2( | u_h |^{q-1} + | \nabla u_h |^{\frac{2(q-1)}{q}} ) \bigr) | p_h | + \bigl( \varphi_3 + c_3( | u_h|^{\frac{q}{2}} + | \nabla u_h | ) \bigr)  | \nabla p_h | 
 \right]\dx \\
& \leq & \bigl( \Vert \varphi_2 \Vert_{L^{\frac{q}{q-1}}} + c_2(\Vert u_h \Vert_{L^q}^{q-1} + \Vert \nabla u_h \Vert_{L^2}^{\frac{2(q-1)}{q}} ) \bigr)  \Vert p_h \Vert_{L^q}   \\
& & + \bigl( \Vert \varphi_3 \Vert_{L^2} + c_3( \Vert u_h \Vert_{L^q}^{\frac{q}{2}} + \Vert \nabla u_h \Vert_{L^2})  \bigr) \Vert  \nabla p_h \Vert_{L^2} \\[2mm]
& \leq & c \bigl( 1+ \Vert u_h \Vert_{H^1}^{q-1}  \bigr)   \Vert p_h \Vert_{H^1} \leq c \bigl( 1 + \Vert f \Vert_{L^2}^{q-1} \bigr) \Vert p_h \Vert_{H^1} \leq c \bigl( 1 + \Vert f \Vert_{L^q}^{q-1} \bigr) \Vert \nabla p_h \Vert_{L^2},
\end{eqnarray*}
{where we also made use of Poincar\'e's inequality for $D$ and the bound on $u_h$.} The estimate for 
$\Vert p_h \Vert_{L^2}$ now follows from another application  of Poincar\'e's inequality.
\end{proof}
\noindent

\noindent
In order to establish the convergence of $\mathcal J_h'(\Omega^k_h)$ we follow  the general procedure outlined in Section 2.2.1 of  \cite{HinPinUlb08}. The following result can be
seen as an analogue of Lemma 2.2 in \cite{HinPinUlb08}, where the uniform continuity of the derivative of the objective functional that is assumed in that result needs to be replaced
by suitable arguments. 

\begin{lemma} \label{decrease}
 Let $\Omega_h=\Phi_h(\Omref) \in \mathcal S_h, \, \mathcal V_{\Phi_h}$ as in \eqref{vhspace} and  $V_h \in \mathcal V_{\Phi_h}$ such that 
\begin{displaymath}
V_h = \argmin \lbrace J_h'(\Omega_h)[W_h] \, | \, W_h \in \mathcal V_{\Phi_h,} \, | DW_h | \leq 1 \mbox{ in } \bar D \rbrace.
\end{displaymath}
Suppose that $\mathcal J_h'(\Omega_h)[V_h] \leq - \epsilon$ for some $\epsilon>0$. Then there exists 
$0< \delta <1$ which only depends on $j, f,D, d, \gamma$ and $\epsilon$ such that
\begin{displaymath}
\mathcal J_h(\Omega_{h,t}) - \mathcal J_h(\Omega_h) \leq \gamma t \mathcal J_h'(\Omega_h)[V_h] \qquad \mbox{ for all } 0 \leq t \leq \delta,
\end{displaymath}
where $\Omega_{h,t} = T_t(\Omega_h)$ and $T_t=\mbox{id}+ t V_h$.
\end{lemma}
\begin{proof} We follow the standard procedure for calculating the shape derivative with special attention on controlling the remainder terms. Recalling the definition of $\mathcal J_h$ we have
\begin{displaymath}
\mathcal J_h(\Omega_{h,t}) = \int_{\Omega_{h,t}} j(\cdot,u_{h,t},\nabla u_{h,t}) \dx,
\end{displaymath}
where $u_{h,t} \in X_{\Omega_{h,t}}$ solves
\begin{displaymath}  
\int_{\Omega_{h,t}} \nabla u_{h,t} \cdot \nabla \eta_{h,t} \dx = \int_{\Omega_{h,t}} f \eta_{h,t} \dx \qquad \forall \eta_{h,t} \in X_{\Omega_{h,t}}.
\end{displaymath}
For $\eta_h \in X_{\Omega_h}$ we have that $\eta_{h,t}:= \eta_h \circ T_t^{-1} \in X_{\Omega_{h,t}}$ and hence
\begin{equation}  \label{hatuht}
\int_{\Omega_{h,t}} \nabla u_{h,t} \cdot \nabla (\eta_h \circ T_t^{-1}) \dx = \int_{\Omega_{h,t}} f \eta_h \circ T_t^{-1} \dx \qquad
\forall  \eta_h \in X_{\Omega_h},
\end{equation}
from which we infer with the help of the transformation rule
\begin{equation} \label{uhtrel}
\int_{\Omega_h} \nabla u_{h,t} \circ T_t \cdot \nabla (\eta_h \circ T_t^{-1}) \circ T_t \, | \mbox{det} D T_t | \dx =
\int_{\Omega_h} f \circ T_t \, \eta_h \, | \mbox{det} D T_t | \dx \qquad \forall  \eta_h \in X_{\Omega_h}.
\end{equation}
Since $| D V_h | \leq 1$ in $\bar D$ we have 
\begin{equation} \label{dif1}
\mbox{det} DT_t -1 = t {\Div} V_h + r_1, \quad \mbox{ with } | r_1 | \leq c t^2,
\end{equation}
where the constant $c$ only depends on $d$. In particular there is $\delta_1>0$ so that $\mbox{det} DT_t >0, 0 \leq t \leq \delta_1$. If
we define $\hat u_{h,t}:= u_{h,t} \circ T_t \in X_{\Omega_h}$ and $A_t:= (DT_t)^{-1} (D T_t)^{- \mathsf{T}} \mbox{det} DT_t$ the relation \eqref{uhtrel} can
 be written in the form
\begin{equation} \label{stateuht}
\int_{\Omega_h} A_t \nabla \hat u_{h,t} \cdot \nabla \eta_h \dx = \int_{\Omega_h} f \circ T_t \, \eta_h \, \mbox{det} D T_t \dx \qquad
\mbox{ for all } \eta_h \in X_{\Omega_h}.
\end{equation}
Thus we have
\begin{eqnarray}
\lefteqn{
\mathcal J_h(\Omega_{h,t}) - \mathcal J_h(\Omega_h) } \nonumber \\
& = & \int_{\Omega_h} \bigl( j(T_t,\hat u_{h,t}, D T_t^{- \mathsf{T}} \nabla \hat u_{h,t}) \, \mbox{det} DT_t - j(\cdot,u_h,\nabla u_h) \bigr) \dx   \nonumber \\
& = & \int_{\Omega_h} j(\cdot,u_h,\nabla u_h) (\mbox{det} DT_t -1 ) \dx + \int_{\Omega_h} \bigl( j(T_t,\hat u_{h,t},DT_t^{- \mathsf{T}} \nabla \hat u_{h,t}) - j(\cdot,u_h,\nabla u_h) \bigr) \dx    \nonumber  \\
& & +  \int_{\Omega_h} \bigl( j(T_t,\hat u_{h,t}, DT_t^{- \mathsf{T}} \nabla \hat u_{h,t}) -j(\cdot,u_h,\nabla u_h) \bigr) ( \mbox{det} DT_t - 1) \dx  \nonumber \\
& = &  \sum_{j=1}^3 T_j.  \label{difj}
\end{eqnarray}
We deduce with the help of \eqref{dif1},  \eqref{jest}, {\eqref{Dembed}} and \eqref{discapriori} that 
\begin{eqnarray}
\lefteqn{
T_1 = t \int_{\Omega_h} j(\cdot,u_h,\nabla u_h) \Div V_h \dx + \int_{\Omega_h} j(\cdot,u_h,\nabla u_h) \, r_1 \dx } \nonumber \\
& \leq &  t \int_{\Omega_h} j(\cdot,u_h,\nabla u_h) \Div V_h \dx + c t^2 \int_{\Omega_h} ( \varphi_1 + c_1 | u_h |^q + c_1 | \nabla u_h |^2) \dx
\nonumber  \\
& \leq & t \int_{\Omega_h} j(\cdot,u_h,\nabla u_h) \Div V_h \dx + c t^2 \bigl( 1+ \Vert u_h \Vert_{H^1}^q \bigr) \leq t \int_{\Omega_h} j(\cdot,u_h,\nabla u_h) \Div V_h \dx + c t^2.
 \label{t1}
\end{eqnarray}
In order to treat $T_2$ we use Taylor's formula and write
\begin{eqnarray*}
\lefteqn{
j(T_t,\hat u_{h,t}, D T_t^{- \mathsf{T}} \nabla \hat u_{h,t}) - j(\cdot,u_h,\nabla u_h) } \\[2mm]
& = &  t j_x() \cdot  V_h + j_u()  (\hat u_{h,t} - u_h) + j_z() \cdot (DT_t^{- \mathsf{T}} \nabla \hat u_{h,t} - \nabla u_h)  \\
& &   + \int_0^1 (1-s) \frac{d^2}{ds^2} \left[ j(\cdot +stV_h,s \hat u_{h,t} +(1-s) u_h,
s  DT_t^{- \mathsf{T}} \nabla \hat u_{h,t} +(1-s) \nabla u_h) \right] ds,
\end{eqnarray*} 
where the first order derivatives of $j$ are evaluated at $(\cdot,u_h,\nabla u_h)$. Thus we have
\begin{eqnarray}
T_2 & = & t  \int_{\Omega_h} j_x(\cdot,u_h,\nabla u_h) \cdot V_h \dx - t \int_{\Omega_h} j_z(\cdot,u_h,\nabla u_h) \cdot DV_h^{\mathsf{T}} \nabla u_h \nonumber \\
& & + \int_{\Omega_h} j_z(\cdot,u_h,\nabla u_h) \cdot \bigl( (DT_t^{-\mathsf{T}} -I + t DV_h^{\mathsf{T}} ) \nabla \hat u_{h,t} + t DV_h^{\mathsf{T}} \nabla (u_h - \hat u_{h,t}) \bigr) \dx \nonumber \\
& & +  \int_{\Omega_h} \bigl( j_u(\cdot,u_h,\nabla u_h)(\hat u_{h,t} -u_h)  + j_z(\cdot,u_h,\nabla u_h) \cdot \nabla (\hat u_{h,t} - u_h) \bigr) \dx \nonumber \\
& & + \int_{\Omega_h}  \int_0^1 (1-s) \frac{d^2}{ds^2} \left[ j(\cdot +stV_h,s \hat u_{h,t} +(1-s) u_h,
s  DT_t^{-\mathsf{T}} \nabla \hat u_{h,t} +(1-s) \nabla u_h) \right] ds \dx \nonumber \\
& = & \sum_{j=1}^5 T_{2,j}. \label{t2}
\end{eqnarray}
Let us begin with the term $T_{2,3}$. Observing that $DT_t^{-\mathsf{T}} = (I+ t D V_h)^{-\mathsf{T}}  = I-t DV_h^{\mathsf{T}} + R_2$ with $|R_2| \leq c t^2$ we deduce with the help
of H\"older's inequality, {\eqref{j2est}, \eqref{Dembed} and \eqref{discapriori}}
\begin{eqnarray}
T_{2,3} & \leq  &   \int_{\Omega_h} |j_z(\cdot,u_h,\nabla u_h) | \bigl( |R_2 | \, | \nabla \hat u_{h,t} | +  t \, | \nabla (\hat u_{h,t} - u_h) | \bigr)  \dx   \nonumber \\
& \leq & c \bigl( \Vert \varphi_3 \Vert_{L^2} + \Vert u_h \Vert_{L^q}^{\frac{q}{2}} + \Vert \nabla u_h \Vert_{L^2} \bigr) \bigl( t^2 \Vert \hat u_{h,t} \Vert_{H^1} + t \Vert \hat u_{h,t} - u_h \Vert_{H^1}
\bigr) \nonumber \\
& \leq &  c \bigl( 1+ \Vert u_h \Vert_{H^1}^{\frac{q}{2}} \bigr) \bigl(  t^2 + c \Vert \hat u_{h,t} - u_h \Vert_{H^1}^2 \bigr) \leq c \bigl(  t^2 + c \Vert \hat u_{h,t} - u_h \Vert_{H^1}^2 \bigr).   \label{t22}
\end{eqnarray}
Next, using  \eqref{discadj}, \eqref{discstate} and \eqref{stateuht} we obtain
\begin{eqnarray*}
T_{2,4} &=& \int_{\Omega_h} \nabla p_h \cdot \nabla ( \hat u_{h,t} - u_h ) \dx = \int_{\Omega_h} \nabla p_h \cdot \nabla \hat u_{h,t} \dx - \int_{\Omega_h} \nabla p_h \cdot \nabla u_h \dx  \\
& = & \int_{\Omega_h} (I- A_t) \nabla p_h \cdot \nabla u_h \dx + \int_{\Omega_h} (I - A_t) \nabla p_h \cdot \nabla ( \hat u_{h,t} - u_h) \dx \\
& & + \int_{\Omega_h}  \bigl( f \circ T_t \, \mbox{det} D T_t -f \bigr) p_h \dx 
 =  \sum_{k=1}^3  \tilde T_k.
\end{eqnarray*}
Recalling that $A_t = (DT_t)^{-1} (DT_t)^{-\mathsf{T}} \mbox{det} DT_t$ it is not difficult to see that
\begin{equation} \label{idmat}
I - A_t = t \bigl( D V_h + D V_h^{\mathsf{T}} - \Div V_h I  \bigr) + R_3, \qquad \mbox{ with } | R_3  | \leq c t^2,
\end{equation}
where $c$ only depends on $d$. Hence
\begin{eqnarray} 
\tilde T_1 & = & t \int_{\Omega_h} \bigl( D V_h + D V_h^{\mathsf{T}} - \Div V_h I  \bigr) \nabla u_h \cdot \nabla p_h \dx 
+ \int_{\Omega_h} R_3 \nabla u_h \cdot \nabla p_h \dx \nonumber \\
& \leq &  t \int_{\Omega_h} \bigl( D V_h + D V_h^{\mathsf{T}} - \Div V_h I  \bigr) \nabla u_h \cdot \nabla p_h \dx + c t^2   \label{t21}
\end{eqnarray}
in view of  \eqref{discapriori}. Next
\begin{equation} \label{t22Dash}
\tilde T_2 \leq c t \Vert \nabla p_h \Vert_{L^2} \Vert \nabla ( \hat u_{h,t} - u_h) \Vert_{L^2} \leq c t \Vert \nabla ( \hat u_{h,t} - u_h) \Vert_{L^2}
\end{equation}
again by \eqref{discapriori}.
In order to deal with $\tilde T_3$ we  write for $x \in \Omega_h$
\begin{displaymath}
 f(T_t(x)) =  f(x) +  t \int_0^1 \nabla f(x+s t V_h(x)) \cdot V_h(x) \ds,
 \end{displaymath}
which, combined with \eqref{dif1} yields
\begin{eqnarray}
f \circ T_t  \, \mbox{det} D T_t -f & = &  ( f \circ T_t -f ) + t f \circ T_t \, \Div V_h + r_1 f \circ T_t   \label{fdif} \\
&=& t \nabla f \cdot V_h + t f \Div V_h  + t \int_0^1 \bigl( \nabla f(\cdot+st V_h) - \nabla f \bigr) \cdot V_h \ds \nonumber \\
& &  + t^2 \int_0^1 \nabla f(\cdot +st V_h) \cdot V_h \ds \, \Div V_h + r_1 f \circ T_t.  \nonumber
\end{eqnarray}
This implies together with \eqref{dif1} and \eqref{discapriori} 
\begin{eqnarray} 
\tilde T_3  & \leq &      t \int_{\Omega_h} \Div( f V_h) p_h \dx  +  c t^2 \Vert p_h \Vert_{L^2}  \Vert f \Vert_{H^1} + c t \Vert p_h \Vert_{L^2}  
\sup_{0 \leq \sigma \leq t} \Vert \nabla f \circ T_\sigma - \nabla f \Vert_{L^2}  \nonumber \\
& \leq & t \int_{\Omega_h} \Div( f V_h) p_h \dx +  c t^2 + c t  \sup_{0 \leq \sigma \leq t} \Vert \nabla f \circ T_\sigma - \nabla f \Vert_{L^2}.  \label{t23}
\end{eqnarray}
Here we have also used that
\begin{equation} \label{Vhest}
| V_h(x)| \leq \mbox{diam}(D) \sup_{y \in D} | D V_h(y) | \leq \mbox{diam}(D), \; x \in D,
\end{equation}
since $V_h=0$ on $\partial D$ and $| DV_h | \leq 1$ in $D$.
Collecting the above terms we have
\begin{eqnarray}
T_{2,4} & \leq &  t \int_{\Omega_h} \bigl( D V_h + D V_h^{\mathsf{T}} - \Div V_h I  \bigr) \nabla u_h \cdot \nabla p_h \dx + 
t \int_{\Omega_h} \Div( f V_h) p_h \dx  \nonumber \\
& & +  c t^2 +c \Vert \hat u_{h,t} - u \Vert_{H^1}^2 + c t  \sup_{0 \leq \sigma \leq t} \Vert \nabla f \circ T_\sigma - \nabla f \Vert_{L^2}. \label{t24}
\end{eqnarray}
Finally, the term $T_{2,5}$ involves a sum of products of  second order partial derivatives of $j$ with $t V_h, \hat u_{h,t} - u_h$ and
$DT_t^{-\mathsf{T}} \nabla \hat u_{h,t}- \nabla u_h$. By way of example we use \eqref{j2est} to estimate for $0 \leq s \leq 1$
\begin{eqnarray*}
\lefteqn{
\int_{\Omega_h} | j_{xz}(\cdot +stV_h,s \hat u_{h,t} +(1-s) u_h,s DT_t^{-\mathsf{T}} \nabla \hat u_{h,t} +(1-s) \nabla u_h) | \, t \,  | V_h| \, | DT_t^{-\mathsf{T}} \nabla \hat u_{h,t} -  \nabla u_h | \dx } \\
& \leq & t \Vert V_h \Vert_{L^\infty} \int_{\Omega_h} \bigl( \varphi_3 \circ (\mbox{id} + st V_h) + | s \hat u_{h,t} +(1-s) u_h |^{\frac{q}{2}} \\
& &  \qquad \qquad \quad  \quad + | s DT_t^{-\mathsf{T}} \nabla \hat u_{h,t} +(1-s) \nabla u_h |
\bigr) | DT_t^{-\mathsf{T}} \nabla \hat u_{h,t} - \nabla u_h |  \dx  \\
& \leq &  c t  \bigl( \Vert \varphi_3 \Vert_{L^2}+ \Vert \hat u_{h,t} \Vert_{L^q}^{\frac{q}{2}} + \Vert u_h \Vert_{L^q}^{\frac{q}{2}} + \Vert \nabla \hat u_{h,t} \Vert_{L^2}
 + \Vert \nabla  u_{h} \Vert_{L^2} \bigr) \Vert DT_t^{-\mathsf{T}} \nabla \hat u_{h,t} - \nabla u_h \Vert_{L^2} \\
& \leq & c t  \bigl( 1+ \Vert \hat u_{h,t} \Vert_{H^1}^{\frac{q}{2}} + \Vert u_h \Vert_{H^1}^{\frac{q}{2}} \bigr) \bigl( \Vert DT_t^{-\mathsf{T}} - I \Vert_{L^\infty} \Vert \hat u_{h,t} \Vert_{H^1} +  \Vert \hat u_{h,t} - u_h \Vert_{H^1} \bigr) \\
& \leq & ct \bigl( t + \Vert \hat u_{h,t} - u_h \Vert_{H^1} \bigr)  \leq c t^2 + c \Vert \hat u_{h,t} - u_h \Vert_{H^1}^2.
\end{eqnarray*}
Arguing in a similar way for the other terms we obtain
\begin{equation} \label{t25}
T_{2,5}   \leq    c t^2 + c \Vert \hat u_{h,t} - u_h \Vert_{H^1}^2,
\end{equation}
so that in conclusion
\begin{eqnarray}
T_2 & \leq &  t  \int_{\Omega_h} j_x(\cdot,u_h,\nabla u_h) \cdot V_h \dx - t \int_{\Omega_h} j_z(\cdot,u_h,\nabla u_h) \cdot DV_h^{\mathsf{T}} \nabla u_h \nonumber \\
& & +  t \int_{\Omega_h} \bigl( D V_h + D V_h^{\mathsf{T}} - \Div V_h I  \bigr) \nabla u_h \cdot \nabla p_h \dx + 
t \int_{\Omega_h} \Div( f V_h) p_h \dx  \nonumber \\
& & +  c t^2 +c \Vert \hat u_{h,t} - u_h \Vert_{H^1}^2 + c t  \sup_{0 \leq \sigma \leq t} \Vert \nabla f \circ T_\sigma - \nabla f \Vert_{L^2}. \label{t2a}
\end{eqnarray}
In order to treat $T_3$ we write
\begin{eqnarray*}
\lefteqn{ j(T_t,\hat u_{h,t}, DT_t^{-\mathsf{T}} \nabla \hat u_{h,t}) -j(\cdot,u_h,\nabla u_h) } \\
& = & \int_0^1 \frac{d}{ds}  j(\cdot +stV_h,s \hat u_{h,t} +(1-s) u_h,s DT_t^{-\mathsf{T}} \nabla \hat u_{h,t} +(1-s) \nabla u_h) ds,
\end{eqnarray*}
use the growth assumptions on $j_x,j_u,j_z$ as well as \eqref{dif1} and derive
\begin{equation} \label{t3}
T_3 \leq ct \bigl( t + \Vert \hat u_{h,t} - u_h \Vert_{H^1} \bigr) \leq c t^2 + c \Vert \hat u_{h,t} - u_h \Vert_{H^1}^2.
\end{equation}
If we insert the estimates \eqref{t1}, \eqref{t2a} and \eqref{t3} into \eqref{difj} and recall \eqref{discsd} we obtain
\begin{equation} \label{difj1}
\mathcal J_h(\Omega_{h,t}) - \mathcal J_h(\Omega_h)  \leq    t \mathcal J_h'(\Omega_h)[V_h]  + ct \bigl( t  +  \sup_{0 \leq \sigma \leq t} \Vert \nabla f \circ T_\sigma - \nabla f \Vert_{L^2} \bigr)
+ c \Vert \hat u_{h,t} - u_h \Vert^2_{H^1}.
\end{equation}
In order to estimate $\Vert \hat u_{h,t} - u_h \Vert_{H^1}$ we combine \eqref{discstate} and \eqref{stateuht} to obtain
\begin{displaymath}
\int_{\Omega_h} A_t  \nabla ( \hat u_{h,t} - u_h) \cdot \nabla \eta_h \dx = \int_{\Omega_h} (I-A_t) \nabla u_h \cdot \nabla \eta_h \dx + \int_{\Omega_h} \bigl(
f \circ T_t  \, \mbox{det} D T_t -f \bigr) \eta_h \dx \qquad \forall \eta_h \in X_{\Omega_h}.
\end{displaymath}
In view of \eqref{idmat} there exists $0< \delta_2 \leq \delta_1$ such that $A_t \xi \cdot \xi \geq \frac{1}{2} | \xi |^2$ for all $\xi \in \mathbb R^d$ and $0 \leq t \leq \delta_2$. Inserting $\eta_h=
\hat u_{h,t}-u_h$ into the above relation and using \eqref{idmat} as well as \eqref{fdif} we infer with the help of Poincar\'e's inequality that 
\begin{eqnarray*}
\frac{1}{2} \int_D |  \nabla (\hat u_{h,t} - u_h) |^2 \dx  & \leq  &c t \Vert \nabla u_h \Vert_{L^2} \Vert \nabla (\hat u_{h,t} - u_h) \Vert_{L^2}
+ c t \Vert  f \Vert_{H^1}  \Vert \hat u_{h,t} -u_h \Vert_{L^2}  \\
& \leq & \frac{1}{4}  \Vert \nabla (\hat u_{h,t} - u_h) \Vert_{L^2}^2 + c t^2,
\end{eqnarray*}
from which we deduce  that $ \Vert \hat u_{h,t} - u_h \Vert_{H^1} \leq ct$. If we insert this bound into
\eqref{difj1} and use that $\mathcal J_h'(\Omega_h)[V_h] \leq - \epsilon$ we obtain
\begin{eqnarray}
\lefteqn{ \hspace{-1.5cm}
\mathcal J_h(\Omega_{h,t}) - \mathcal J_h(\Omega_h) 
 \leq     t \mathcal J_h'(\Omega_h)[V_h] + c t \bigl(t  +  \sup_{0 \leq \sigma \leq t} \Vert \nabla f \circ T_\sigma - \nabla f \Vert_{L^2} \bigr) }  \nonumber  \\
& \leq &  \gamma t \mathcal J_h'(\Omega_h)[V_h] +  c t \bigl( t  + \sup_{0 \leq \sigma \leq t} \Vert \nabla f \circ T_\sigma - \nabla f \Vert_{L^2} \bigr) -  (1- \gamma) \epsilon t. 
\label{final}
\end{eqnarray}
There exists $0< \delta \leq \delta_2$ such that $\sup_{0 \leq \sigma \leq t} \Vert \nabla f \circ T_\sigma - \nabla f \Vert_{L^2} \leq 
\frac{1}{2c} (1-\gamma)   \epsilon$ for $0 \leq t \leq  \delta$. This can be seen as usual by approximating $f_{x_i}$ by a continuous function $g_i$,
using the uniform continuity of $g_i$ on $\bar D$ and noting that by \eqref{Vhest} 
\begin{displaymath}
| T_\sigma(x) - x | = \sigma | V_h(x) | \leq \mbox{diam}(D) t,  \quad 0 \leq \sigma \leq t, \, x \in D.
\end{displaymath}
By choosing $\delta $ smaller if necessary we can achieve in addition that $\delta  \leq \frac{1}{2c} (1-\gamma)   \epsilon$. Inserting these bounds into \eqref{final} we obtain the result of the theorem.
\end{proof}

\noindent
We are now in position to prove our first convergence result.

\begin{thm} \label{conv1}
Let $(\Phi_h^k)_{k \in \mathbb N_0} \subset \hatXh$ and $(\Omega^k_h=\Phi^k_h(\Omref))_{k \in \mathbb N_0} \subset \mathcal S_h$ be the sequences generated by Algorithm \ref{alg:Steepest}. 
Then: \\[3mm]
(i) $\Vert \mathcal J_h'(\Omega^k_h) \Vert \rightarrow 0$ as $k \rightarrow \infty$. \\[2mm]
(ii) If $\sup_{k \in \mathbb N_0}  | (D \Phi^k_h)^{-1} | \leq C$, then there exists a subsequence
$(\Phi^{k_\ell}_h)_{\ell \in \mathbb N}$, which converges in $W^{1,\infty}(D)$ to  a mapping $\Phi_h \in \hatXh$ and   $\Omega_h:= \Phi_h(\Omref)$
 is a stationary point of $\mathcal J_h$, i.e.~satisfies $\mathcal J_h'(\Omega_h)[V_h]=0$ for all $V_h \in \mathcal V_{\Phi_h}$.
\end{thm}
\begin{proof} (i) Since $\mathcal J_h(\Omega^{k+1}_h) \leq \mathcal J_h(\Omega^k_h)$ and
\begin{eqnarray*}
\mathcal J_h(\Omega^k_h) & \geq &  - \int_{\Omega^k_h} | j(\cdot,u^k_h,\nabla u^k_h) | \dx \geq - \int_{\Omega^k_h} ( \varphi_1+ c_1 \bigl(  | u^k_h |^q + | \nabla u^k_h |^2 \bigr)  \dx  \\
& \geq & - c \bigl( 1+ \Vert u^k_h \Vert^q_{H^1}  \bigr)  \geq -c,
\end{eqnarray*}
we infer that $\lim_{k \rightarrow \infty} \mathcal J_h(\Omega^k_h)=:\beta \in \mathbb R$ exists. Then
\begin{displaymath}
\sum_{k=0}^{\infty} \bigl( \mathcal J_h(\Omega^{k}_h) - \mathcal J_h(\Omega^{k+1}_h) \bigr) = \mathcal J_h(\Omega^0_h) -\beta < \infty,
\end{displaymath}
so that
\begin{equation} \label{jkzero}
 \mathcal J_h(\Omega^{k}_h) - \mathcal J_h(\Omega^{k+1}_h) \rightarrow 0 \quad \mbox{ as } k \rightarrow \infty.
 \end{equation}
Suppose that   $\Vert \mathcal J_h'(\Omega^k_h) \Vert \nrightarrow 0$. Then there exists $\epsilon>0$ and a subsequence $(\Omega^{k_\ell}_h)_{\ell \in \mathbb N}$
such that $\Vert \mathcal J_h'(\Omega^{k_\ell}_h) \Vert \geq \epsilon$ for all $\ell \in \mathbb N$. In view of the definition of $V^k_h$  we infer that
\begin{equation} \label{leps}
\mathcal J_h'(\Omega^{k_\ell}_h)[V^{k_\ell}_h] = -  \|\mathcal{J}_h'(\Omega_h^{k_l})\| \leq - \epsilon \quad \mbox{ for all } \ell \in \mathbb N,
\end{equation}
and Lemma \ref{decrease} yields the existence of $\delta>0$ which is independent of $\ell \in \mathbb N$ such that
\begin{displaymath}
\mathcal J_h\bigl( (\mbox{id}+ t V^{k_\ell}_h)(\Omega^{k_\ell}_h) \bigr)  - \mathcal J_h(\Omega^{k_\ell}_h) \leq \gamma t \mathcal J_h'(\Omega^{k_\ell}_h)[V^{k_\ell}_h] 
 \qquad \mbox{ for all } 0 \leq t \leq \delta.
\end{displaymath}
Therefore we have that the Armijo step size satisfies $t_{k_\ell} \geq \frac{\delta}{2}$ for all $\ell \in \mathbb N$ from which we deduce with the help of
\eqref{leps} that
\begin{displaymath}
\mathcal J_h(\Omega^{k_\ell}_h) - \mathcal J_h(\Omega^{k_{\ell}+1}_h) \geq -\gamma t_{k_\ell} \mathcal J_h'(\Omega_h^{k_\ell})[V_h^{k_\ell}]  \geq \gamma t_{k_\ell} \epsilon
 \geq \gamma \frac{\delta}{2} \epsilon \qquad  \mbox{ for all } \ell \in \mathbb N
\end{displaymath}
contradicting \eqref{jkzero}. \\[2mm]
(ii) Since $\hatXh$ is a subset of a finite--dimensional space and  the sequence $(\Phi^k_h)_{k \in \mathbb N_0}$ is bounded (recall that $\Phi^k_h(\bar D) = \bar D$), there exists
a subsequence, again denoted by $(\Phi^k_h)_{k \in \mathbb N_0}$ and $\Phi_h \in C^0(\bar D, \mathbb R^d)$ such that $\Phi^k_h \rightarrow \Phi_h$ in $W^{1,\infty}(D)$. 
Furthermore, as $\sup_{k \in \mathbb N_0}  | (D \Phi^k_h)^{-1} | \leq C$ we have
\begin{displaymath}
| x_1 -x_2 | \leq C | \Phi^k_h(x_1) - \Phi_h^k(x_2) | \quad \mbox{ for all } x_1,x_2 \in \bar D, k \in \mathbb N_0
\end{displaymath}
from which we infer that $\Phi_h$ is injective by letting $k \rightarrow \infty$. Thus, $\Phi_h \in \mathcal U_h$. Let us show that
$\Omega_h:= \Phi_h(\Omref)$ is a stationary point of $\mathcal J_h$.
As  most of the necessary arguments have appeared in some form in the proof of Lemma \ref{decrease} we only sketch the main ideas. Let us define
$T_k:= \Phi_h \circ (\Phi^k_h)^{-1}$. Clearly $T_k \rightarrow \mbox{id}$ in $W^{1,\infty}(D,\mathbb R^d)$ as $k \rightarrow \infty$. Furthermore, let
$\hat u^k_h:= u^k_h \circ T_k^{-1} \in X_{\Omega_h}, \hat p^k_h:= p^k_h \circ T_k^{-1} \in X_{\Omega_h}$, where $u_h, p_h$ and $u^k_h,p^k_h$ are
the discrete state and adjoint state in $\Omega_h$ and $\Omega^k_h$ respectively. One can show similarly as
above that $\hat u^k_h \rightarrow u_h, \hat p^k_h \rightarrow p_h$ in $H^1(\Omega_h)$. Let us fix $V_h \in \mathcal V_{\Phi_h}$ and consider the terms
that appear in the formula \eqref{discsd} for $\mathcal J_h'(\Omega_h)[V_h]$. For the first integral we write
\begin{eqnarray*}
\lefteqn{
\int_{\Omega_h} j(\cdot,u_h,\nabla u_h) \Div V_h  \dx  } \\
&= &   \int_{\Omega_h} \bigl( j(\cdot,u_h,\nabla u_h) - j(\cdot,\hat u^k_h,\nabla \hat u^k_h) \bigr) \Div V_h  \dx + \int_{\Omega_h} j(\cdot,\hat u^k_h,\nabla
\hat u^k_h) 
\Div V_h \dx  \\
& = & \int_{\Omega_h} \int_0^1 j_u(\cdot,s u_h + (1-s) \hat u^k_h,s \nabla u_h + (1-s) \nabla \hat u^k_h) (u_h - \hat u^k_h) \ds \,  \Div V_h  \dx \\
& & + \int_{\Omega_h} \int_0^1  j_z(\cdot,s u_h + (1-s) \hat u^k_h,s \nabla u_h + (1-s) \nabla \hat u^k_h) \cdot \nabla (u_h - \hat u^k_h) \ds \,  \Div V_h  \dx  \\
& & + \int_{\Omega^k_h} j(\cdot,u^k_h,\nabla u^k_h) (\Div V_h) \circ T_k \,
\mbox{det} D T_k  \dx \\
& = & \int_{\Omega^k_h} j(\cdot,u^k_h,\nabla u^k_h) \Div( V_h \circ T_k)  \dx + o(1)
\end{eqnarray*}
since $\hat u^k_h \rightarrow u_h$ in $H^1(\Omega_h)$ and $T_k \rightarrow \mbox{id}$ in $W^{1,\infty}(D,\mathbb R^d)$. If we argue in a similar way for the other
terms in \eqref{decrease} we obtain that
\begin{displaymath}
\mathcal J_h'(\Omega_h)[V_h] = \mathcal J_h'(\Omega^k_h)[V_h \circ T_k] + o(1).
\end{displaymath}
Observing that $\Vert  D(V_h \circ T_k) \Vert_{L^\infty} \leq c$ we deduce with the help of (i) that $\mathcal J_h'(\Omega_h)[V_h]=0$. Since $V_h \in \mathcal V_{\Phi_h}$ was arbitrary, the
result follows.
\end{proof}

\section{Convergence of stationary shapes}\label{sec:ConvergenceOfStationaryPoints}
{We now investigate the convergence of stationary shapes when the discretisation parameter $h$ tends to zero. To begin with we first introduce two appropriate convergence measures for shapes.}

\subsection{Hausdorff convergence and Mosco--convergence}

Before we investigate the convergence of a sequence of stationary shapes we introduce two important concepts.  The Hausdorff complementary distance of  two open sets $\Omega_1, \Omega_2 \subset D$ is
defined as
\begin{displaymath}
\rho_H^c(\Omega_1,\Omega_2):= \max_{x \in \bar D} | d_{\complement \Omega_1}(x) - d_{\complement \Omega_2}(x) |,
\end{displaymath}
where $d_{\complement \Omega}(x):= \inf\{|x-y| : y \in \bar D \setminus \Omega\}$ for all $x \in D$,
and we say that $(\Omega_k)_{k \in \mathbb N}$  converges to $\Omega$ in the sense of the Hausdorff complementary metric if $\rho_H^c(\Omega_k,\Omega) \rightarrow 0,
k \rightarrow \infty$. Here  $\Omega_k, \Omega$ are open subsets of $D$. Since our optimisation problem is constrained by the elliptic boundary value problem \eqref{state} a stronger convergence
concept is required that ensures continuity of \eqref{state} with respect to $\Omega$ in an appropriate sense. For an open set $\Omega \subset D$ we shall view 
$H^1_0(\Omega)$ as a closed subspace of $H^1_0(D)$ by associating with each element $u \in H^1_0(\Omega)$ its extension by zero $e_0(u) \in H^1_0(D)$.

\begin{defn} Let $\Omega_k, \Omega$ be open subsets of $D$.  We say that  $(\Omega_k)_{k \in \mathbb N}$ converges to $\Omega$ in the sense of Mosco if 
the following conditions hold: \\[2mm]
(i) For every $u \in H^1_0(\Omega)$ there exists a sequence $(u_k)_{k \in \mathbb N}$ with $u_k \in H^1_0(\Omega_k)$ such that $e_0(u_k) \rightarrow e_0(u)$ in
$H^1_0(D)$. \\[2mm]
(ii) If $(u_{k_\ell})_{\ell \in \mathbb N}$ is a sequence with $u_{k_\ell} \in H^1_0(\Omega_{k_\ell})$ and $e_0(u_{k_\ell}) \rightharpoonup v$ in $H^1_0(D)$, then $v \in H^1_0(\Omega)$. 
\end{defn}

\noindent
In order to formulate a corresponding convergence result we recall that the 2--capacity of a set $A \subset U $ relative to an open bounded set $U$ is defined by 
\begin{displaymath}
\mbox{cap}(A,U):= \inf \left\lbrace \int_U | \nabla v |^2 \dx \, | \, v \in H^1_0(U), v \geq 1 \mbox{ a.e. in a neighbourhood of } A \right\rbrace.
\end{displaymath}

\noindent
\begin{defn} Let $\Omega \subset D$ be open. \\[2mm] 
a) We say that $\complement \Omega$  satisfies a capacity density condition,  if there exist $\alpha>0,r_0>0$ such that 
\begin{equation} \label{capdens}
  \frac{\mbox{cap}(B_r(x) \cap \complement \Omega,B_{2r}(x))}{\mbox{cap}(B_r(x),B_{2r}(x))} \geq \alpha \qquad \mbox{ for all } 0<r < r_0
  \mbox{ and all } x \in \partial \Omega.
\end{equation}
We denote by $\mathcal O_{\alpha,r_0}$ the collection of all open subsets $\Omega \subset D$ that satisfy \eqref{capdens}
with $\alpha>0,r_0>0$. 
\end{defn}

\begin{thm} \label{compact0} Let $(\Omega_k)_{k \in \mathbb N}$ be a sequence of open subsets of $D$ belonging to $\mathcal O_{\alpha,r_0}$, which
  converges in the sense of the Hausdorff complementary metric to an open set $\Omega$. Then $(\Omega_k)_{k \in \mathbb N}$ converges to $\Omega$ in the sense of Mosco.
\end{thm}
\begin{proof} In view of Theorem 3.4.12 in \cite{HenPie18} the sequence $(\Omega_k)_{k \in \mathbb N}$ $\gamma$--converges to $\Omega$. However, according
to Proposition 3.5.5 \cite{HenPie18} $\gamma$-convergence  and Mosco convergence are equivalent which implies the result.
\end{proof}

\noindent
In order to make use of the above result in the setting considered in our paper we require the following lemma.

\begin{lemma} \label{compact}
Let $\Omega=\Phi(\Omref)$ for some bilipschitz map $\Phi: \bar D \rightarrow \bar D$ satisfying
\begin{displaymath}
M^{-1}| x - y | \leq | \Phi(x)-\Phi(y) | \leq M | x-y | \qquad \mbox{ for all } x,y \in D.
\end{displaymath}
Then there exist $\alpha>0, r_0>0$ depending on $\Omref, d$ and $M$ such that $\Omega \in \mathcal O_{\alpha,r_0}$.
\end{lemma}
\begin{proof} Let $r_0:= M s_0$ with $s_0$ as in \eqref{corkscrew}. For $x \in \partial \Omega$ there exists $\hat x \in \partial \Omref$ such that
$x=\Phi(\hat x)$. Given $0<r<r_0$, we let $s=\frac{r}{M} \in (0,s_0)$ and
choose $\hat y \in B_s(\hat x)$ such that $B_{\lambda s}(\hat y) \subset B_s(\hat x) \cap \complement \Omref$ according to \eqref{corkscrew}. Then $y=\Phi(\hat y)$ satisfies
\begin{displaymath}
| y - x | = | \Phi(\hat y) - \Phi(\hat x) | \leq M | \hat y - \hat x| < M s =r,
\end{displaymath}
so that $y \in B_r(x)$. We claim that 
\begin{equation} \label{cs1}
B_{\frac{\lambda r}{M^2}}(y) \subset B_r(x) \cap \complement \Omega.
\end{equation}
To see this, let $z \in B_{\frac{\lambda r}{M^2}}(y)$, say $z=\Phi(\hat z)$ for some $\hat z \in \bar D$.
Then,
\begin{displaymath}
 | \hat z - \hat y | \leq M | \Phi(\hat z) - \Phi(\hat y) |  = M  | z - y | < \frac{\lambda r}{M}  =   \lambda s.
\end{displaymath}
Hence $\hat z \in B_{\lambda s}(\hat y)$ and therefore $\hat z \in B_s(\hat x) \cap \complement \Omref$. Then $z=\Phi(\hat z) \in \complement \Omega$ with $| z- x|  \leq M | \hat z - \hat x | < Ms =r$ implying 
\eqref{cs1}. Since $B_{2r}(x) \subset B_{3r}(y)$ we deduce from \eqref{cs1}
\begin{displaymath}
  \frac{\mbox{cap}(B_r(x) \cap \complement \Omega,B_{2r}(x))}{\mbox{cap}(B_r(x),B_{2r}(x))} \geq   \frac{\mbox{cap}(B_{\frac{\lambda r}{M^2}}(y),B_{3r}(y))}{\mbox{cap}(B_r(x),B_{2r}(x))} \geq \alpha,
\end{displaymath}
where $\alpha$ only depends on $\lambda,M$ and $d$. Here, the last inequality can be shown as in the proof of Theorem 6.31 in \cite{HKM}.
\end{proof}

\subsection{Convergence of discrete stationary shapes}

\noindent
Let  $(\mathcal { \hat T}_h)_{0< h \leq h_0}$ be a regular family of triangulations of $\bar D$ in the sense that
there exists $\sigma>0$ such that
\begin{equation}  \label{regular}
  \frac{h_{\hat T}}{\rho_{\hat T}} \leq \sigma \quad \forall \hat T \in \mathcal { \hat T}_h, \; 0< h \leq h_0.
\end{equation}
Here $h_{\hat T}$ is the diameter of $\hat T$ and $\rho_{\hat T}$ the diameter of
the largest ball contained in
$\hat T$. 
We consider the corresponding sequence of discrete shape functionals $\mathcal J_h: \mathcal S_h \rightarrow \mathbb R$ given by \eqref{discshapeopt}. In what follows we assume the existence
of a sequence $(\Omega_h)_{0<h \leq h_0}$ such that $\Omega_h = \Phi_h(\Omref)$ for some  $\Phi_h \in \hatXh$ and \\[2mm]
(A1) $\forall 0< h \leq h_0 \;  \forall V_h \in \mathcal V_{\Phi_h}:  \quad \mathcal J_h'(\Omega_h)[V_h] = 0$; \\[2mm]
(A2) $\exists M>1 \; \forall 0< h \leq h_0 \;  \forall x,y \in D: \quad M^{-1}| x - y | \leq | \Phi_h(x)-\Phi_h(y) | \leq M | x-y |$. \\[2mm]
Assumption (A1)  states that the sequence $(\Omega_h)_{0<h \leq h_0}$  is a sequence of stationary points, while (A2) can be interpreted as a compactness property of these sets.
Such a condition appears in \cite{BarWac20}, where it  occurs in the convergence analysis  for a sequence of discrete minima.

\noindent

\begin{thm} \label{conv3a} Let $\Omega_h=\Phi_h(\Omref)$, where $\Phi_h \in \hatXh$ and $(\Phi_h)_{0 < h \leq h_0}$ satisfies (A2).  
  Then there exists a sequence $(h_k)_{k \in \mathbb N}$
  with $\lim_{k \rightarrow \infty} h_k=0$ and a map $\Phi \in \mathcal U$ such that: \\[3mm]
  (i) $\Phi_{h_k} \rightarrow \Phi$ uniformly in $\bar D$, $\rho_H^c( \Omega_{h_k},\Omega) \rightarrow 0$ as $ k \rightarrow \infty$, where $\Omega=\Phi(\Omref)$; \\[2mm]
(ii) $e_0(u_{h_k}) \rightarrow e_0(u)$ in $H^1_0(D)$, where $u_{h_k} \in X_{\Omega_{h_k}}$ solves \eqref{discstate}, $u \in H^1_0(\Omega)$ solves
\eqref{state}; \\[2mm]
(iii) $e_0(p_{h_k}) \rightarrow e_0(p)$ in $H^1_0(D)$, where $p_{h_k} \in X_{\Omega_{h_k}}$ solves \eqref{discadj}, $p \in H^1_0(\Omega)$ solves
\eqref{adj}.
\end{thm}
\begin{proof}
(i)  In view of (A2) the sequences $(\Phi_h)_{0,h \leq h_0}$ and $(\Phi^{-1}_h)_{0< h \leq h_0}$ are uniformly bounded and uniformly equicontinuous, so that the Arzela--Ascoli theorem implies that there exists a sequence $h_k \rightarrow 0$ and  functions
$\Phi, \Psi \in W^{1,\infty}(D, \mathbb R^d)$ such that $\Phi=\mbox{id}$ on $\partial D$ and 
\begin{displaymath}
\Phi_{h_k} \rightarrow \Phi, \quad \Phi_{h_k}^{-1} \rightarrow \Psi  \mbox{ in } C^0(\bar D, \mathbb R^d) \; \mbox{ as } \; k \rightarrow \infty.
\end{displaymath}
Clearly, $\Phi:\bar D \rightarrow \bar D$ is bilipschitz so that $\Phi \in \mathcal U$. 
Let $\Omega=\Phi(\Omref)$. We claim that
\begin{equation} \label{hausdorff}
\rho_H^c( \Omega_{h_k},\Omega) \leq \Vert \Phi_{h_k} - \Phi \Vert_{C^0(\bar D, \mathbb R^d)}.
\end{equation}
To see this, let $x \in \bar D$ and choose $ y \in \bar D \setminus \Omega$ such that $d_{\complement \Omega}(x) = | x -  y|$. In view of
the definition of $\Omega$ there exists $z \in \bar D \setminus \Omref$ such that $y= \Phi(z)$. Then, $y_k:= \Phi_{h_k}(z) \in \bar D \setminus \Omega_{h_k}$
and therefore 
\begin{displaymath}
d_{\complement \Omega_{h_k} }(x) - d_{\complement \Omega}(x) \leq | x - y_k | - | x-y | \leq | y_k - y |
\leq \Vert \Phi_{h_k} - \Phi \Vert_{C^0(\bar D,\mathbb R^d)}.
\end{displaymath}
By exchanging the roles of $\Omega_{h_k}$ and $\Omega$ we deduce \eqref{hausdorff}, which implies that
$\rho_H^c( \Omega_{h_k},\Omega) \rightarrow 0$ as $ k \rightarrow \infty$.  \\[2mm]
(ii) Our line of argument is similar as in \cite{CZ06}. Since $(e_0(u_{h_k}))_{k \in \mathbb N}$ is bounded in $H^1_0(D)$ we may assume after possibly extracting a subsequence that there exists $u^* \in H^1_0(D)$ such that 
$e_0(u_{h_k}) \rightharpoonup u^*$ in $H^1_0(D)$.  In view of  (A2) and Lemma \ref{compact} we may apply Theorem \ref{compact0} so that 
$(\Omega_{h_k})_{k \in \mathbb N}$ converges to $\Omega$ in the sense of Mosco. In particular we infer that $u^* \in H^1_0(\Omega)$. In order to see that 
$u^*$ is the solution of \eqref{state} we fix $\eta \in C^\infty_0(\Omega)$ and set $K:=\mbox{supp} \eta$. Since $\rho_H^c( \Omega_{h_k},\Omega) \rightarrow 0$, Proposition 2.2.17 in \cite{HenPie18} implies that
there exists $k_0 \in \mathbb N$ such that $K \subset \Omega_{h_k}$ for all $k \geq k_0$. Let us denote by $I_{h_k} \eta \in X_{\Omega_{h_k}}$ the standard Lagrange interpolation
of $\eta$, for which we have
\begin{displaymath}
\Vert \eta - I_{h_k} \eta \Vert_{H^1_0(D)} \leq c h_k \Vert \eta \Vert_{H^2(D)}.
\end{displaymath}
Here we have used the fact that 
the family of triangulations $\bigl( \Phi_h(\hat T) \, | \, \hat T \in \hat{ \mathcal T}_h \bigr)_{0< h \leq h_0}$ is regular with
\begin{equation} \label{regular1} 
  \frac{h_T}{\rho_T} \leq M^2 \frac{h_{\hat T}}{\rho_{\hat T}} \leq M^2 \sigma, \quad T=\Phi_h(\hat T), \,
    0 < h \leq h_0,
\end{equation}
where $\sigma$ appears in \eqref{regular} and we have again applied (A2). Thus $I_{h_k} \eta \rightarrow \eta$ in $H^1_0(D)$ and by inserting $I_{h_k}\eta$ into \eqref{discstate} 
we obtain
\begin{displaymath}
\int_\Omega \nabla u^* \cdot \nabla \eta \dx = \lim_{k \rightarrow \infty} \int_D \nabla e_0(u_{h_k}) \cdot \nabla I_{h_k} \eta \dx = \lim_{k \rightarrow \infty} \int_D f I_{h_k} \eta \dx = \int_\Omega f \eta \dx.
\end{displaymath}
Hence $u^*=e_0(u)$, where $u$ is the solution of \eqref{state}.  A standard argument (see Corollary 3.2.2 in \cite{HenPie18}) then shows that $e_0(u_{h_k}) \rightarrow
e_0(u)$ in $H^1_0(D)$ as $k \rightarrow \infty$. \\
(iii) In the same way as in (ii) we infer that there exists $p^* \in H^1_0(\Omega)$ such that $e_0(p_{h_k})
\rightharpoonup p^*$ in $H^1_0(D)$ as $k \rightarrow \infty$.
{After possibly extracting a further subsequence we may assume that $e_0(u_{h_k}) \rightarrow e_0(u)$ and
$\nabla e_0(u_{h_k}) \rightarrow \nabla e_0(u)$ almost everywhere in $D$.  We claim that 
\begin{eqnarray}
j_u(\cdot,e_0(u_{h_k}),\nabla e_0(u_{h,k})) &\rightarrow & j_u(\cdot,e_0(u), \nabla e_0(u)) \quad \mbox{ in } L^{\frac{q}{q-1}}(D), \label{convph1} \\
j_z(\cdot,e_0(u_{h_k}),\nabla e_0(u_{h,k})) & \rightarrow & j_z(\cdot,e_0(u), \nabla e_0(u)) \quad \mbox{ in } L^2(D,\mathbb R^d).\label{convph2}
\end{eqnarray}
In order to show \eqref{convph2} we set $f_k:=| j_z(\cdot,e_0(u_{h_k}),\nabla e_0(u_{h,k})) - j_z(\cdot,e_0(u), \nabla e_0(u)) |^2$. Clearly, $f_k \rightarrow 0$ a.e.~in $D$, while
\eqref{j2est} implies that 
\begin{displaymath}
f_k \leq c \bigl( \varphi_3^2 + | u_{h_k} |^q + | \nabla u_{h_k} |^2 + | u |^q + | \nabla u |^2 \bigr)=: g_k.
\end{displaymath}
We have that $g_k \rightarrow g:=c \bigl( \varphi_3^2 + 2 | u |^q + 2| \nabla u |^2 \bigr)$ a.e.~in $D$ as well as $\int_D g_k dx \rightarrow \int_D g dx$ as $k
\rightarrow \infty$, so that the generalised Lebesgue dominated convergence theorem yields \eqref{convph2}.
The relation \eqref{convph1} is proved in the same way.
With the help of \eqref{convph1} and \eqref{convph2}} we obtain similarly as above that  $p^*=e_0(p)$, where $p$ is the solution of \eqref{adj} and then again
$e_0(p_{h_k}) \rightarrow e_0(p)$ in $H^1_0(D)$ as $ k \rightarrow \infty$. 
\end{proof}

\noindent
We can now examine the convergence of a sequence of discrete stationary points as $h \rightarrow 0$. 

\begin{thm} \label{conv3} Suppose that  $(\Omega_h)_{0 <h \leq h_0}$ satisfies (A1) and (A2). 
Then there exists a sequence $(h_k)_{k \in \mathbb N}$ with $\lim_{k \rightarrow \infty} h_k=0$ and an open set $\Omega \Subset D$ such that $\rho_H^c( \Omega_{h_k},\Omega) \rightarrow 0$ as $ k \rightarrow \infty$.
Furthermore,  $\Omega$ is a stationary  point for $\mathcal J$ on $\mathcal{S}$.  
\end{thm}
\begin{proof}
 We infer from Theorem \ref{conv3a} that there exists a sequence $(h_k)_{k \in \mathbb N}$ with $\lim_{k \rightarrow \infty} h_k=0$
 and a bilipschitz map $\Phi\colon \bar D \rightarrow \bar D$ such that for $\Omega=\Phi(\Omref)$ we have
 \begin{equation} \label{limits}
\rho_H^c( \Omega_{h_k},\Omega) \rightarrow 0, \; \;  e_0(u_{h_k}) \rightarrow e_0(u) \mbox{ in }  H^1_0(D), \; \; e_0(p_{h_k}) \rightarrow e_0(p)  \; \;\mbox{ in }  H^1_0(D)\mbox{ as } k \rightarrow \infty
\end{equation}
where $u_{h_k}$, $u$, $p_{h_k}$, and $p$ are as in Theorem \ref{conv3a}.
In order to show that $\Omega$ is a stationary point for $\mathcal J$ on $S$ we first claim that
\begin{equation} \label{indicator}
\chi_{\Omega_{h_k}} \rightarrow \chi_{\Omega} \quad \mbox{ a.e. in  } D.
\end{equation}
Since $\Phi$ is bilipschitz and $ \partial \Omref$ has measure 0 we infer that the same is true for
$\partial \Omega = \Phi(\partial \Omref)$, so that it is sufficient to prove that $\chi_{\Omega_{h_k}}(x) \rightarrow \chi_{\Omega} (x)$ for all $x \in \Omega \cup D \setminus \bar \Omega$. 
To begin, Corollary 1 in Chapter 6, Section 4 of \cite{DelZol01} implies that 
$\chi_{\Omega_{h_k}}(x) \rightarrow \chi_{\Omega}(x)$ for all $x \in 
\Omega$. Next, let $x \in D \setminus \bar \Omega$. We claim that there exists $k_0 \in \mathbb N$ such that $x \in D \setminus \bar \Omega_{h_k}$
for all $k \geq k_0$. Otherwise there is a subsequence $(k_\ell)_{\ell \in \mathbb N}$ and $y_{k_\ell} \in \overline{\Omref}$ such that $\Phi_{h_{k_\ell}}(y_{k_\ell})=x$
for all $\ell \in \mathbb N$. By passing to a further subsequence we may assume that $y_{k_\ell} \rightarrow y$ for some $y \in \overline{\Omref}$, which
together with the uniform convergence of $(\Phi_{h_k})_{k \in \mathbb N}$ to $\Phi$ 
implies that $\Phi(y)=x$, a contradiction. Therefore $\lim_{k \rightarrow \infty} \chi_{\Omega_{h_k}}(x) = 0 = \chi_{\Omega}(x)$ and  \eqref{indicator} holds.  \\
Let us  fix $V \in W^{1,\infty}_0(D,\mathbb R^d)$ and set $V_k:= I_{h_k} V \in \mathcal V_{\Phi_{h_k}}$. 
 We may assume after possibly extracting a further subsequence that
\begin{equation} \label{weakconv}
V_k \rightarrow V \; \mbox{ in } L^\infty(D;\mathbb R^d) \quad \mbox{ and } \quad D V_k \overset{*}{\rightharpoonup} DV \mbox{ in }  L^\infty(D; \mathbb R^{d \times d}).
\end{equation}
As $\Omega_{h_k}$ is a stationary point for $\mathcal J_{h_k}$ we have
\begin{eqnarray}
0 & = & \mathcal J_{h_k}'(\Omega_{h_k})[V_k] = \int_{\Omega_{h_k}} \Bigl( 
\bigl( DV_k + D V_k^{\mathsf{T}} - \Div V_k I \bigr)  \nabla u_{h_k} \cdot \nabla p_{h_k}
+ \Div( f V_k) p_{h_k}   \Bigr) \dx \nonumber \\
&  & +  \int_{\Omega_{h_k}} \Bigl( j(\cdot,u_{h_k},\nabla u_{h_k}) \Div V_k +  j_x(\cdot,u_{h_k},\nabla u_{h_k}) \cdot V_k - j_z(\cdot,u_{h_k},\nabla u_{h_k}) \cdot DV_k^{\mathsf{T}} \nabla u_{h_k} \Bigr) \dx  \nonumber  \\
& =: & A_k+B_k. \label{akbk}
\end{eqnarray}
In view of  \eqref{indicator} and \eqref{limits} it is not difficult to verify that 
\begin{displaymath}
\chi_{\Omega_{h_k}} \nabla e_0(u_{h_k}) \cdot \nabla e_0(p_{h_k}) \rightarrow \chi_{\Omega} \nabla e_0(u) \cdot \nabla e_0(p) \mbox{ in } L^1(D),
\end{displaymath}
which together with \eqref{weakconv} yields
\begin{eqnarray*}
A_k & = &  \int_D \chi_{\Omega_{h_k}}  \Bigl(  \bigl( DV_k + D V_k^{\mathsf{T}} - \Div V_k I \bigr)  \nabla e_0(u_{h_k}) \cdot \nabla e_0(p_{h_k}) + \Div( f V_k) e_0(p_{h_k})   \Bigr) \dx \\
& \rightarrow & \int_D \chi_{\Omega}  \Bigl( \bigl( DV + D V^{\mathsf{T}} - \Div V I \bigr)  \nabla e_0(u) \cdot \nabla e_0(p)
+ \Div( f V) p   \Bigr) \dx \\
& = & \int_\Omega  \Bigl(  \bigl( DV + D V^{\mathsf{T}} - \Div V I \bigr)  \nabla u \cdot \nabla p
+ \Div( f V) p   \Bigr) \dx.
\end{eqnarray*} 
Similarly we have
\begin{displaymath}
B_k \rightarrow \int_{\Omega} \bigl( j(\cdot,u,\nabla u) \Div V  + j_x(\cdot,u,\nabla u)  V - j_z(\cdot,u,\nabla u) \cdot DV^{\mathsf{T}} \nabla u \bigr) \dx.
\end{displaymath}
Passing to the limit in \eqref{akbk} we deduce that  $\mathcal J'(\Omega)[V]=0$. 
\end{proof}

\begin{remark}\label{rem:penalty}
We briefly describe how our analysis can be generalized to a setting in which an additional constraint is imposed on the admissible sets.
In particular, we consider a volume constraint, so that 
\begin{displaymath}
\tilde{\mathcal S} = \lbrace \Omega \subset D \, | \, \Omega = \Phi(\hat \Omega) \mbox{ for some } \Phi \in \mathcal U, | \Omega| =m_0 \rbrace,
\end{displaymath}
where $0<m_0 < |D|$ is a given constant. In this case we consider the modified functional 
\begin{displaymath}
\tilde{\mathcal J}_h:\mathcal S_h \rightarrow \mathbb R, \; \; \tilde J_h(\Omega_h):= \mathcal J_h(\Omega_h)+  \mathcal A_h(\Omega_h), \mbox{ where }
\mathcal A_h(\Omega_h) = \frac{\mu_h}{2} \bigl( | \Omega_h | -m_0 \bigr)^2, 
\end{displaymath}
and $(\mu_h)_{0 < h \leq h_0}$ is a sequence of real numbers satisfying $\lim_{h \searrow 0} \mu_h = \infty$. It is not difficult 
to verify that the results of Section \ref{sec:convergenceOfAlgorithm} still hold for $\tilde{ \mathcal J_h}$. Next, suppose that  $(\Omega_h)_{0<h \leq h_0}$ is a sequence of
stationary points of $\tilde{\mathcal J_h}$ satisfying (A2). Arguing as in the proof of Theorem 4.4 one obtains a sequence $(\Omega_{h_k})_{k \in \mathbb N}$ and an open set
$\Omega \Subset D$ such that \eqref{limits} holds.  
In order to show that  $| \Omega | =m_0$ we choose an open set $U \Subset D$ such that $\Omega_{h_k} \subset U$ for all $k \in \mathbb N$ and a
function $\bar V \in C^2_0(D, \mathbb R^d)$ such that $\bar V(x)=x$ for all $x \in U$. Then,  $\bar V_k:= I_{h_k} \bar V \in
\mathcal V_{\Phi_{h_k}}$ satisfies
 $\Vert \bar V_k - \bar V \Vert_{L^\infty} \leq ch_k$, $\Vert D \bar V_k \Vert_{L^\infty} \leq c$ as well as $ \bar V_k(x) = x$ on $\Omega_{h_k}$. Since $\tilde{\mathcal J}_{h_k}'(\Omega_{h_k})[\bar V_k]=0$ 
and $\Div \bar V_k = d$ on $\Omega_{h_k}$ we obtain with the help of \eqref{discsd}, our assumptions on $j$ and  \eqref{discapriori} that
\begin{displaymath}
 \mu_{h_k} d   | \Omega_{h_k} |   \, \big|  \, | \Omega_{h_k}| - m_0 \big|  =  \big|  \mu_{h_k} (|  \Omega_{h_k}| - m_0)  \int_{\Omega_{h_k}} {\Div} \bar V_k \dx \big| = |  \mathcal A_h'(\Omega_{h_k})[\bar V_k]|  =| - \mathcal J_h'(\Omega_{h_k})[\bar V_k]  | \leq c
 \end{displaymath}
 for all $k \in \mathbb N$.
 Observing that $\mu_{h_k} \rightarrow \infty$ and $| \Omega_{h_k} | \rightarrow | \Omega | = | \Phi(\Omref)| >0$ in view of \eqref{indicator}, we deduce that $|\Omega| =m_0$ so that $\Omega \in \tilde{\mathcal S}$. Furthermore, $\Omega$ can be shown to be a stationary
 point of $\mathcal J$ in the sense that $\mathcal J'(\Omega)[V]=0$ for all $V \in W^{1,\infty}_0(D)$ with $\int_{\Omega} {\Div}V \dx =0$.
\end{remark}

\section{Numerical experiments}\label{sec:experiments}
The numerical experiments we provide here show experimental evidence of the convergence we prove, alongside observing any possible rates of convergence.
For the implementation of finite element methods, we will utilise DUNE \cite{duneReference}, particularly the python bindings \cite{DunePython1,DunePython2}.
The initial grid is constructed with \texttt{pygmsh}, \cite{pygmsh}.

Notice that the Hausdorff complementary metric requires the distance function for our provided shape, this is not so trivial to construct; as such, we make use of the construction in \cite{DecHerHin23} as an approximation.
Let $\{x_i\}_{i=1}^N$ be the vertices of the triangulation of $\mathcal{T}_h$ and let $\{y_i\}_{i=1}^n$ be the vertices which lie on the boundary of $\Omega_h$, then we set
\begin{equation}
    d^h_{\complement \Omega_h}(x_i) = \min_{j =1,...,n} |x_i-y_j| \mbox{ if } x_i \in \Omega_h, \mbox{ otherwise } 0.
\end{equation}
We then calculate our discrete Hausdorff complementary distance to be given by
\begin{equation}\label{eq:DHCD}
 \rho_h(\Omega_h,\Omega^*) = \max_{ i =1,\ldots,N } | d_{\complement \Omega^*}(x_i) - d^h_{\complement \Omega_h}(x_i)|.
\end{equation}
In the experiments provided, we will consider shape optimisation problems with known, simple, minimisers.
In particular we will know the explicit form of the complementary distance function $d_{\complement \Omega^*}$.
This allows us to measure the quantities of interest and compare.

We will also measure the (maximum) radius ratio of the initial and final grids, this appears in \cite{IglSturWec18}, for example.
On each cell, this quantity is closely related to the left hand side of \eqref{regular}.
The radius ratio $\sigma$ of a triangle $T$ is given by
\begin{equation}
    \sigma (T) := \frac{r_{T}}{2 \rho_{T}},
\end{equation}
where $r_{T}$ is the radius of the smallest ball which contains $T$ and $\rho_{T}$ is the radius of the largest ball contained in $T$.
It holds that $\sigma\geq 1$ and $\sigma(\hat{T}) = 1$ if and only if $\hat{T}$ is equilateral.
For the initial grid, we will measure $\tilde{\sigma}_0 := \max_{\hat T \in \mathcal{\hat T}_h} \sigma(\hat T)$ and on the final grid, defined by $\Phi_h$, $\tilde{\sigma}_f := \max_{ \hat{T} \in \mathcal{\hat{T}}_h} \sigma(\Phi_h(\hat T))$.

\subsection{Direction of steepest descent construction}
Throughout this work, we have made use of $V_h$, a direction of steepest descent.
While it is known that such a direction exists, by compactness in a finite dimensional space, the construction of it is not necessarily trivial.
As in \cite{DecHerHin23}, we will make use of the Alternating Direction Method of Multipliers (ADMM) approach to approximate a solution.
For a given $\Omega_h := \Phi_h(\Omref)$, let
\begin{equation}
    \mathcal{Q}_{\Phi_h} := \{ q_h \in L^2(D;\R^{d \times d}) : q_h|_T \in P_0(T;\R^{d\times d}),\, T = \Phi_h(\hat{T}),\, \hat{T} \in \hat {\mathcal{T}}_h\}
\end{equation}
be the space of piecewise constant $d\times d$ matrix valued finite elements subordinate to the triangulation induced by $\Phi_h$.
In addition for given $\tau>0$, we consider the Lagrangian $\mathcal{L}_\tau \colon \mathcal{V}_{\Phi_h}\times \mathcal{Q}_{\Phi_h} \times \mathcal{Q}_{\Phi_h} \to \R$ given by:
\begin{equation}
    \mathcal{L}_\tau (V_h,q_h;\lambda_h) := \int_D \left( \lambda_h : (DV_h -q) + \frac{\tau}{2}(DV_h-q_h):(DV_h-q_h)\right) \dx + \mathcal{J}_h'(\Omega_h)[V_h].
\end{equation}
Given $V_h^0 \in \mathcal{V}_{\Phi_h}$, $\lambda_h^0 \in \mathcal{Q}_{\Phi_h}$, and $tol >0$, the  algorithm is then given by
\begin{alg}[ADMM]\label{alg:ADMM}
\ \\ [2mm]
0. Let $R = \infty$ \\[2mm]
For $k = 0,1,2,\ldots$: \\[2mm]
1. If $R < tol$, then stop. \\[2mm]
2. Find $q_h^{k+1} = \argmin \{ \mathcal{L}_\tau (V_h^k,q_h;\lambda_h^k) : q_h \in \mathcal{Q}_{\Phi_h},\, |q_h| \leq 1 \} $. \\[2mm]
3. Find $V_h^{k+1} = \argmin \{ \mathcal{L}_\tau (V_h,q_h^{k+1};\lambda_h^k) : V_h \in \mathcal{V}_{\Phi_h} \} $. \\[2mm]
4. Set $\lambda_h^{k+1} = \lambda_h^{k} + \tau (DV_h^{k+1} - q_h^{k+1})$. \\[2mm]
5. Update $R = \left( \|\lambda_h^{k+1} - \lambda_h^{k}\|_{L^2(D;\R^{d\times d})}^2 + \| DV_h^{k+1} - DV_h^{k}\|_{L^2(D;\R^{d\times d})}^2 \right)^{\frac{1}{2}}$.
\end{alg}
There exist variants in which one may adapt $\tau$ to reduce the number of steps required to achieve a given tolerance, see \cite{BarMil20}.
Such an adaptive variant is used in the numerical experiments.

\subsection{Experiments}
For the numerical experiments presented, we will consider a cascading approach.
In the experiments, we run the described algorithm for up to 15 steps, or until the Armijo step length satisfies $t_k\leq 2^{-11}$, perform a congruent refinement, and start the algorithm with $\Phi_{h/2}^0 = \Phi_h^{k^*}$, where $k^* = \min(15, \inf\{ k \in \mathbb{N} : t_{k-1} \leq 2^{-11}\})$.
For the third experiment, we wish to measure some form of convergence, as such it is reasonable to continue to an appropriate convergence criteria, rather than stopping at some ad-hoc number of steps.
We will again consider the mesh converged when the Armijo step length satisfies $t_k \leq 2^{-11}$.
However, the mesh will be saved at shape $k^*$ as described above, and for the refinement, continue with $\Phi_{h/2}^0 = \Phi_h^{k^*}$.
{We expect this cascading approach to be useful for the efficient calculation of optimal shapes.}

We will fix the domain $D = (-2,2)^2$.

\subsubsection{Experiment 1}\label{sec:exp1}
For this experiment, we consider $j(\cdot,u,z):= \frac{1}{2} (u-u_d)^2$, where we choose $u_d(x) = \frac{4}{\pi}-|x|^2$, for the data for the Poisson problem, we take $f = 1$.
A locally optimal shape is expected to be $\Omega^* := B(0,\frac{4}{\sqrt{3\pi}})$, the ball of radius $\frac{4}{\sqrt{3\pi}}$ at the origin, which has energy $\mathcal{J}(\Omega) = \frac{128}{27 \pi^2}$.
{This may be found by assuming symmetry i.e., the solution is a ball of radius $r>0$.
One may then find the above critical radius and energy using calculus in one dimension.}
We choose $\Omref = (-1,1)^2$.
The initial mesh is displayed on the left of Figure \ref{fig:exp1:domains}, with the hold all in blue, and the initial domain in red.

In Figure \ref{fig:exp1:energy}, we see the energy and the discrete Hausdorff complementary distance \eqref{eq:DHCD} for the experiment along the shape iterates.
\begin{figure}
    \centering
    \includegraphics[width = .45\linewidth]{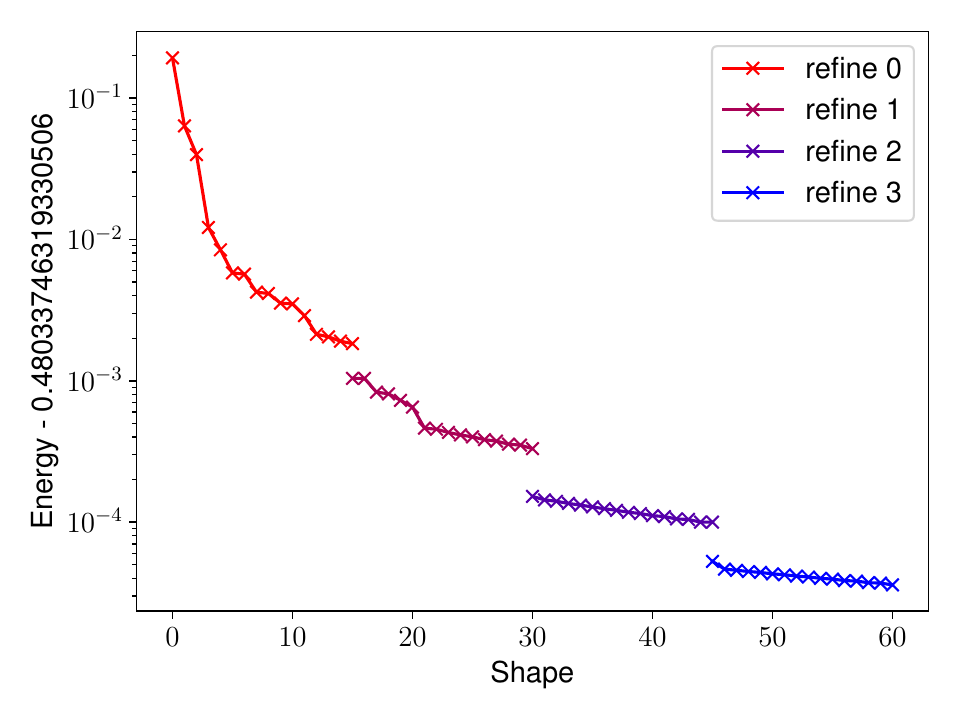}
    \hfill
    \includegraphics[width = .45\linewidth]{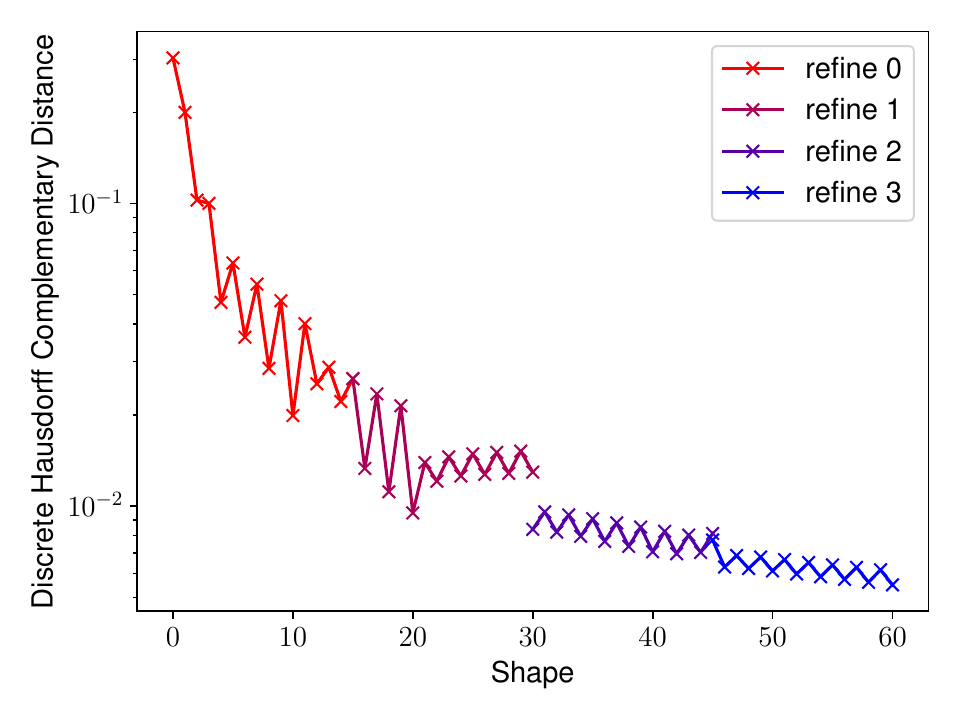}
    \caption{On the left is the energy and on the right discrete Hausdorff complementary distance along the shape iterates for the experiment in Section \ref{sec:exp1}.
        We see that the energy is reducing along the shapes, jumping to a lower energy when the mesh is refined and the distance decreases on average.}
    \label{fig:exp1:energy}
\end{figure}
For convergence of our scheme, in Theorem \ref{conv3}, we require that $\|D\Phi_h\|_{L^\infty}$ and $\|D\Phi_h^{-1}\|_{L^\infty}$ are bounded; the value of these along the iterations is found in Figure \ref{fig:exp1:DPhi}
\begin{figure}
    \centering
    \includegraphics[width = .45\linewidth]{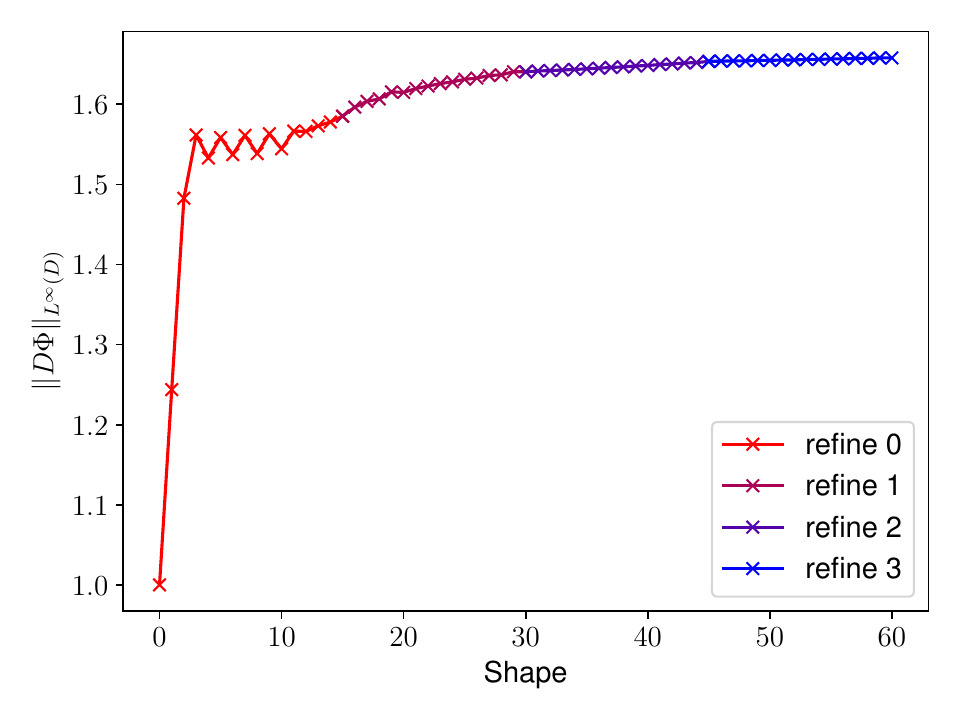}
    \hfill
    \includegraphics[width= .45\linewidth]{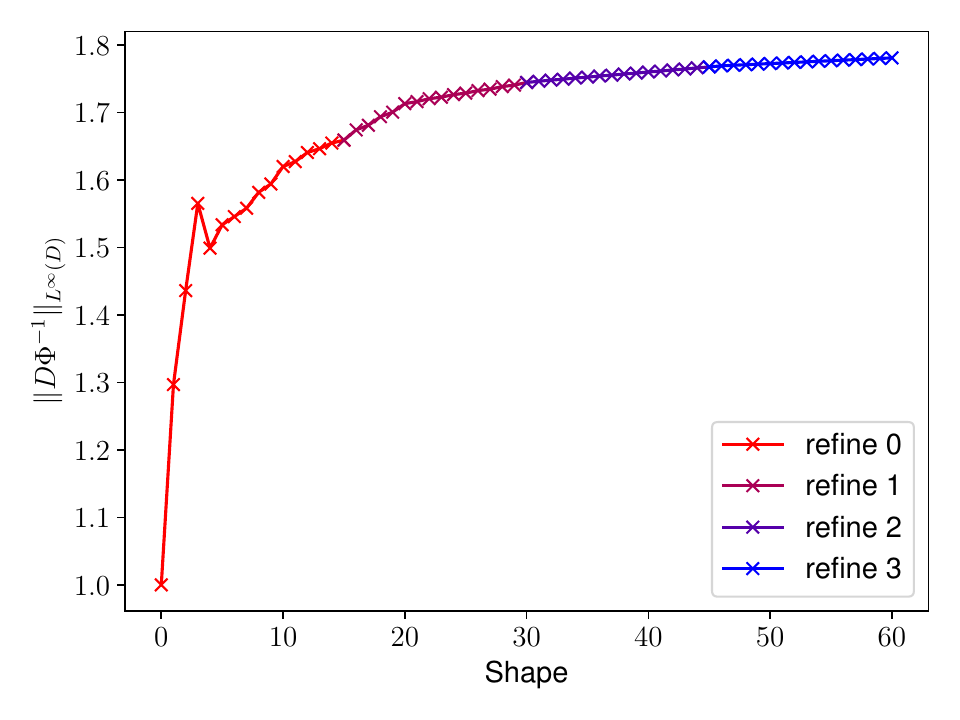}
    \caption{Values for $\|D\Phi_h\|_{L^\infty}$ (left) and $\|D\Phi_h^{-1}\|_{L^\infty}$ (right) along the shape iterates for the experiment in Section \ref{sec:exp1}.
    We see that the values of $\|D\Phi_h\|_{L^\infty}$ and $\|D\Phi_h^{-1}\|_{L^\infty}$ abruptly jump up at the start, and only slowly increase later on.}
    \label{fig:exp1:DPhi}
\end{figure}
The final domains are given on the right of Figure \ref{fig:exp1:domains}.
    
\begin{figure}
    \centering
    \begin{minipage}{.49\linewidth}
        \includegraphics[width = .9 \linewidth]{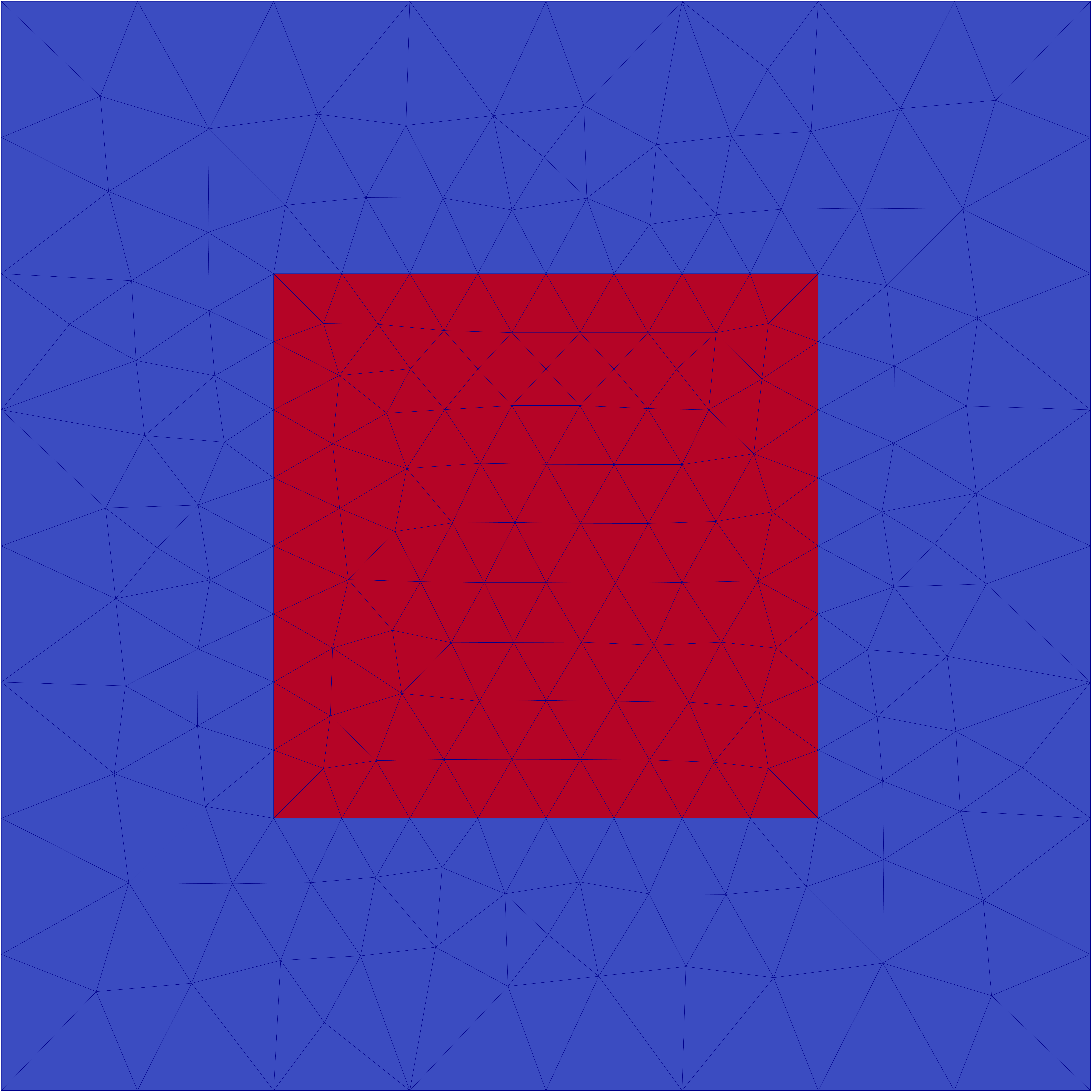}
    \end{minipage}
    \begin{minipage}{.49\linewidth}
        \newlength{\imagewidth}
        \settowidth{\imagewidth}{\includegraphics{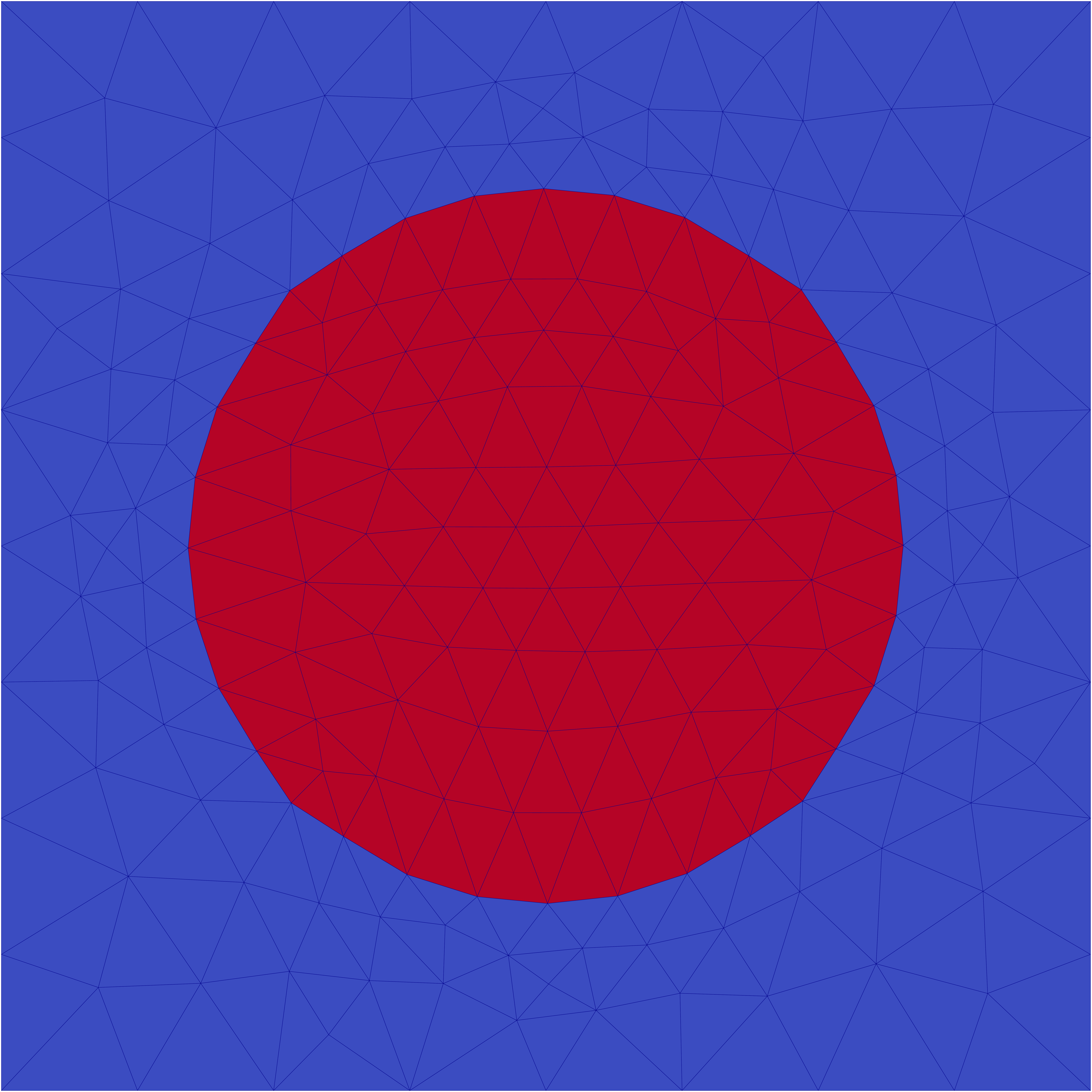}}
        \includegraphics[trim = 0 .5\imagewidth{} .5\imagewidth{} 0, clip,width = .45\linewidth]{experiment_1_refine=0_final_red_blue.pdf}
        \includegraphics[trim = .5\imagewidth{} .5\imagewidth{} 0 0, clip,width = .45\linewidth]{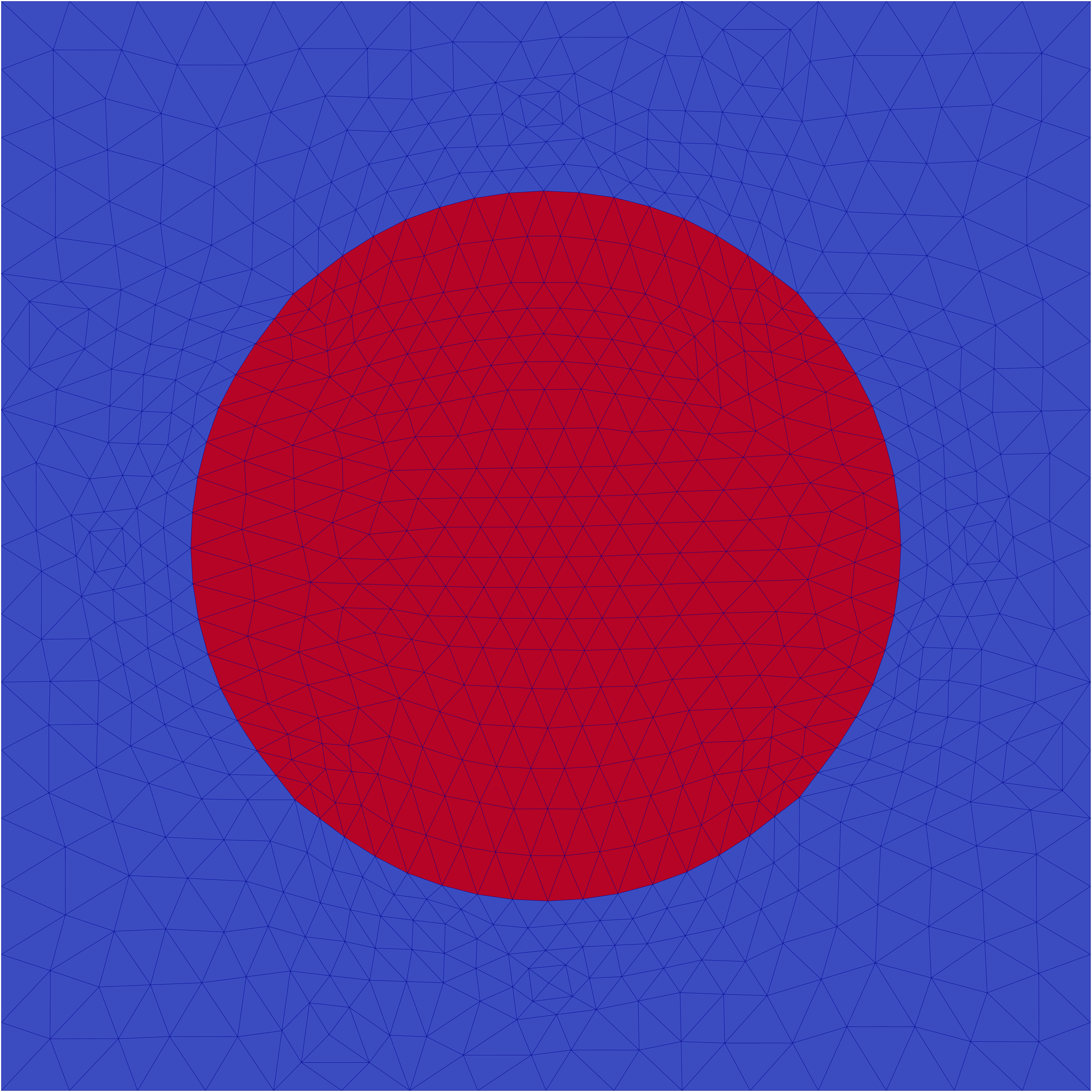}
    
    \noindent
        \includegraphics[trim = 0 0 .5\imagewidth{} .5\imagewidth{}, clip,width = .45\linewidth]{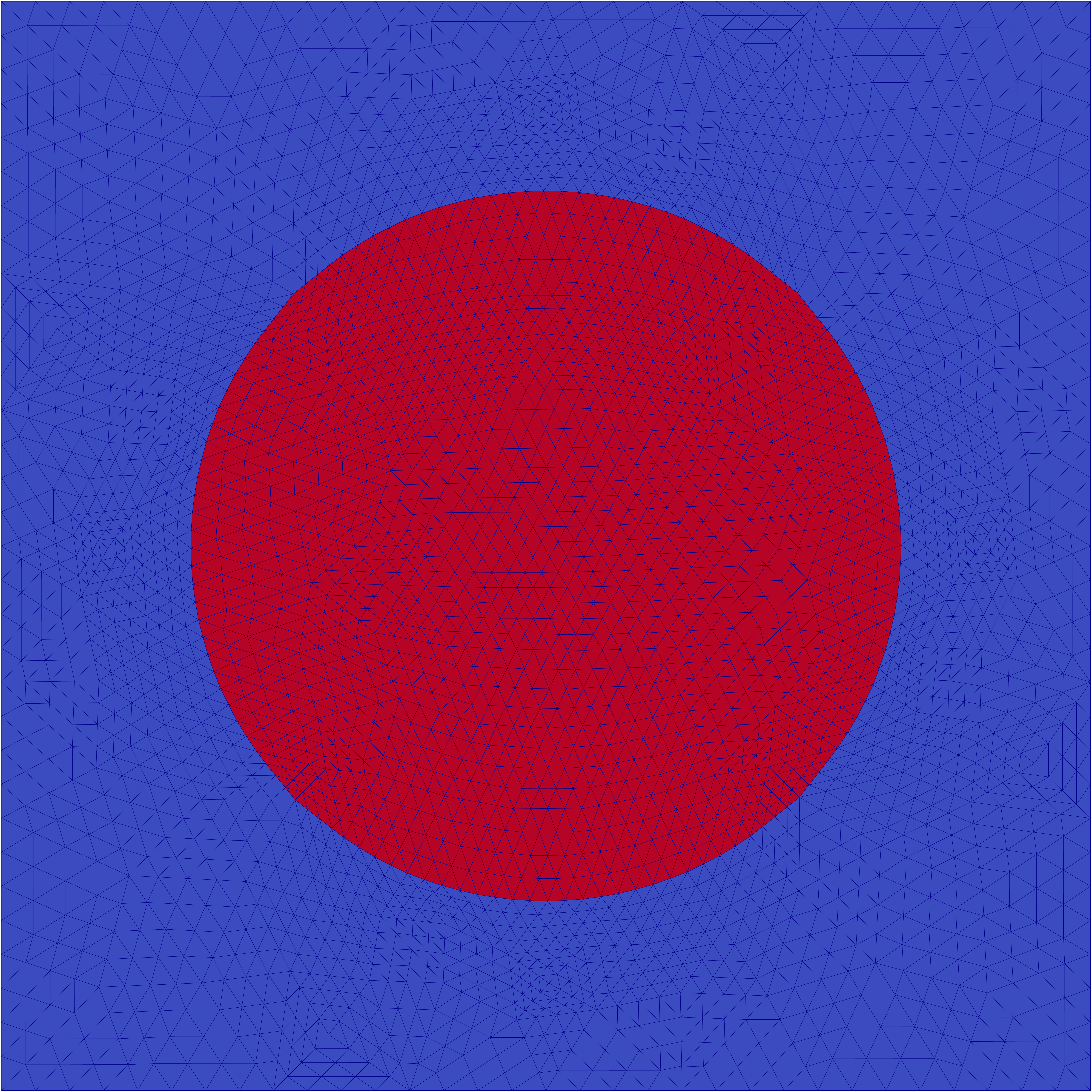}
        \includegraphics[trim = .5\imagewidth{} 0 0 .5\imagewidth{}, clip,width = .45\linewidth]
        {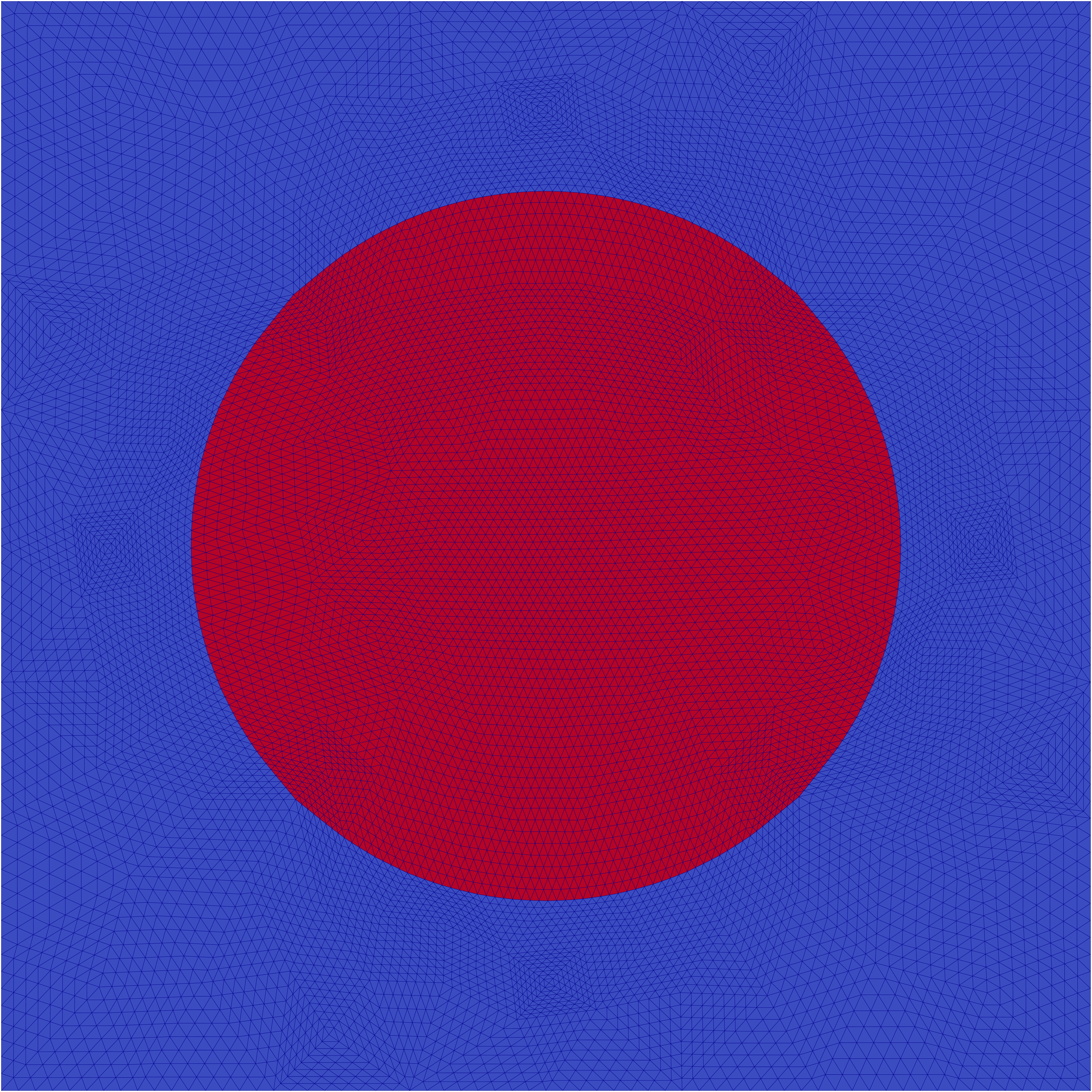}
    \end{minipage}
    \caption{Final domains for the experiment in Section \ref{sec:exp1}, refinements increasing from top left to bottom right.
    Taking the maximum over all triangles, the radius ratio for the initial grid is $\tilde{\sigma}_0\approx 1.633917$ and for the final, most fine, grid is $\tilde{\sigma}_f \approx 1.651944$.
    }
    \label{fig:exp1:domains}
\end{figure}

\subsubsection{Experiment 2}\label{sec:exp2}
For this experiment, we consider $j(\cdot,u,z):= \frac{1}{2} (u-u_d)^2$, where we choose
\begin{equation*}
    u_d(x) = 5 \frac{ -\pi |x|^2\ln(4) + 3 \ln( |x|^2) + 3 \ln(\pi) + \ln(4) }{\pi \ln(256)},
\end{equation*}
for the data for the Poisson problem, we take $f = 5$.
Notice that $u_d$ has zeros on $|x| = \frac{1}{\sqrt{\pi}},\, \frac{2}{\sqrt{\pi}}$ and that $-\Delta u_d = f$.

The optimal shape is expected to be $\Omega^* := B(0,\frac{2}{\sqrt{\pi}}) \setminus \overline{B(0,\frac{1}{\sqrt{\pi}})}$, with $\mathcal{J}(\Omega^*) = 0$.
{As in the first experiment, this is again calculated using axi-symmetric arguments.}
Notice that this is not a simply connected optimal shape, which may require some topology optimisation.
We choose the initial domain given by an approximation of $\Omref = B(0,1.4) \setminus \overline{B(0,0.7)}$.
{Let us note that, without prior knowledge of the topology of the domain, e.g.~starting with a ball of radius $1$, the domains heads towards a non-axi-symmetric shape, which may possibly end up in a degenerate minimiser.
A combined shape and topology optimisation may prove useful in such a setting.
The development of this is work in preparation.}

The initial mesh is displayed on the left of Figure \ref{fig:exp2:domains}, with the hold all in blue, and the initial domain in red.

In Figure \ref{fig:exp2:energy}, we see the energy and the discrete Hausdorff complementary distance \eqref{eq:DHCD} for the experiment along the iterations.
\begin{figure}
    \centering
    \includegraphics[width = .45\linewidth]{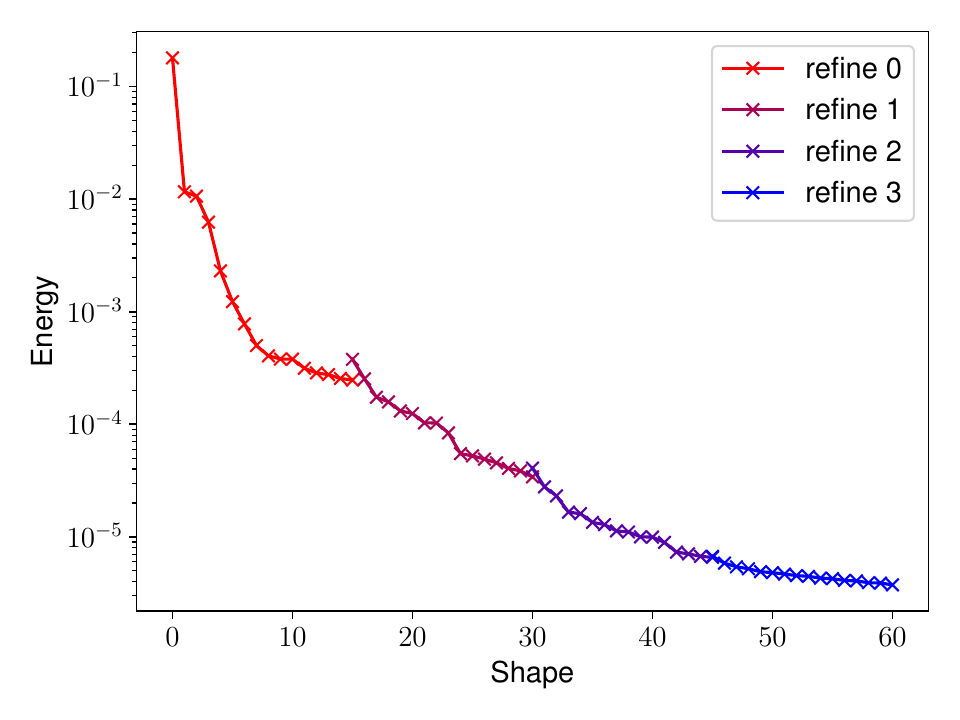}
    \hfill
    \includegraphics[width = .45\linewidth]{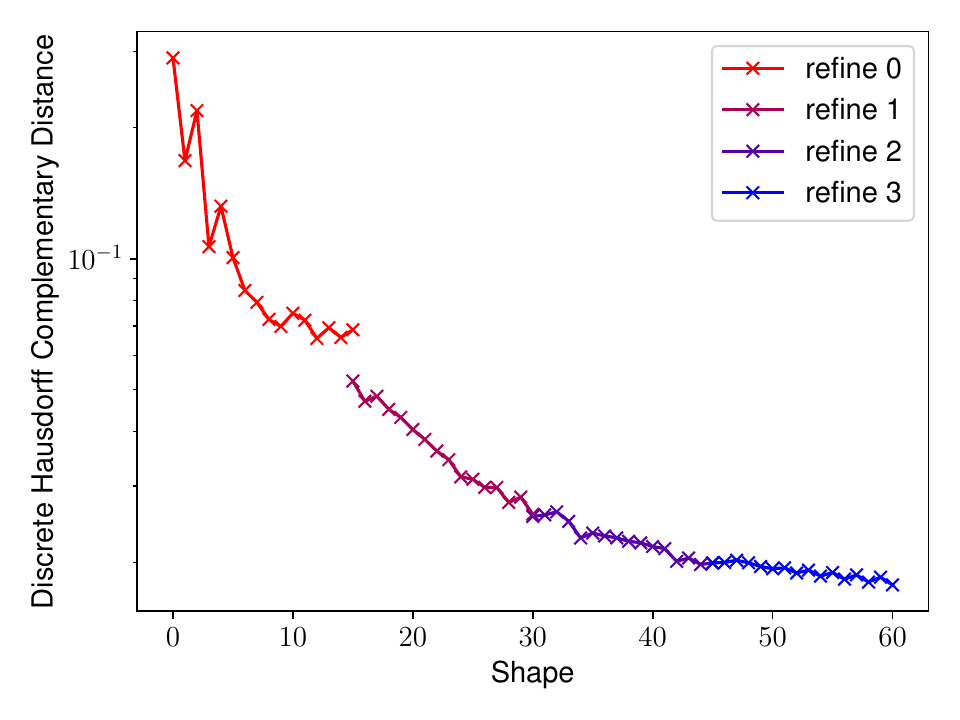}
    \caption{On the left is the energy and on the right discrete Hausdorff complementary distance along the shape iterates for the experiment in Section \ref{sec:exp2}.
        We see that the energy is decreasing, jumping up as the mesh is refined.
        However, the energy recovers only a few steps later to a value smaller than the coarser mesh.
        The distance is reducing on average along the iterations, rarely increasing}
    \label{fig:exp2:energy}
\end{figure}
The value of $\|D\Phi_h\|_{L^\infty}$ and $\|D\Phi_h^{-1}\|_{L^\infty}$ along the iterations is found in Figure \ref{fig:exp2:DPhi}.
\begin{figure}
    \centering
    \includegraphics[width = .45\linewidth]{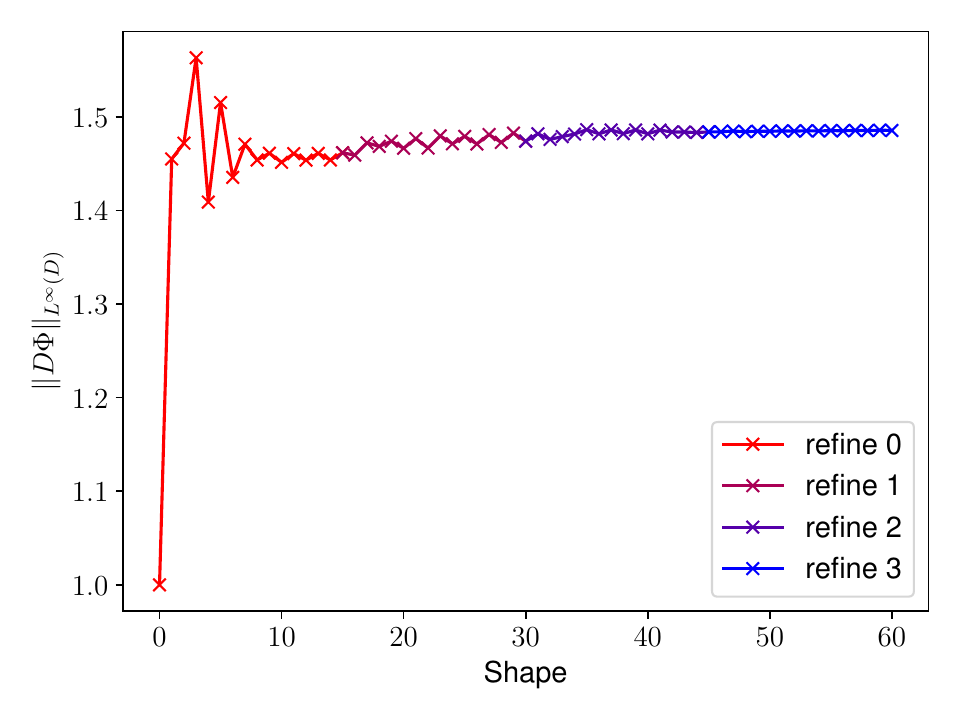}
    \hfill
    \includegraphics[width= .45\linewidth]{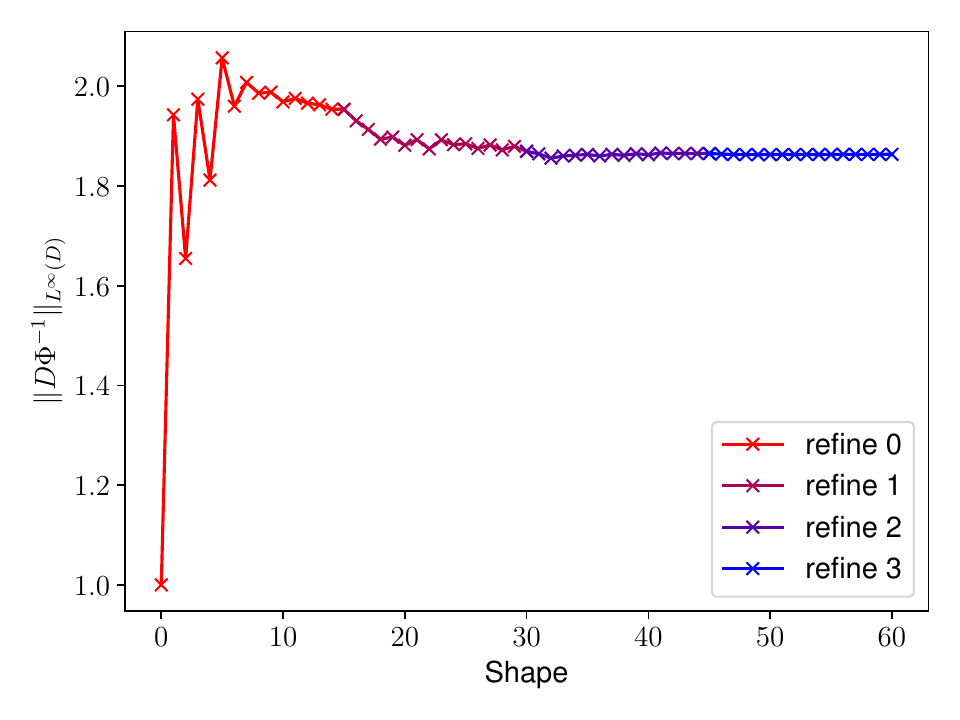}
    \caption{Values for $\|D\Phi_h\|_{L^\infty}$ (left) and $\|D\Phi_h^{-1}\|_{L^\infty}$ (right) along the shape iterates for the experiment in Section \ref{sec:exp2}.
    We see again that the values of $\|D\Phi_h\|_{L^\infty}$ and $\|D\Phi_h^{-1}\|_{L^\infty}$ abruptly jump up at the start.
    Later on, $\|D\Phi_h^{-1}\|_{L^\infty}$ appears to plateau and $\|D\Phi_h\|_{L^\infty}$ increases only very slightly.}
    \label{fig:exp2:DPhi}
\end{figure}
The final domains are given on the right of Figure \ref{fig:exp2:domains}.

\begin{figure}
    \centering
    \begin{minipage}{.49\linewidth}
        \includegraphics[width = .9 \linewidth]{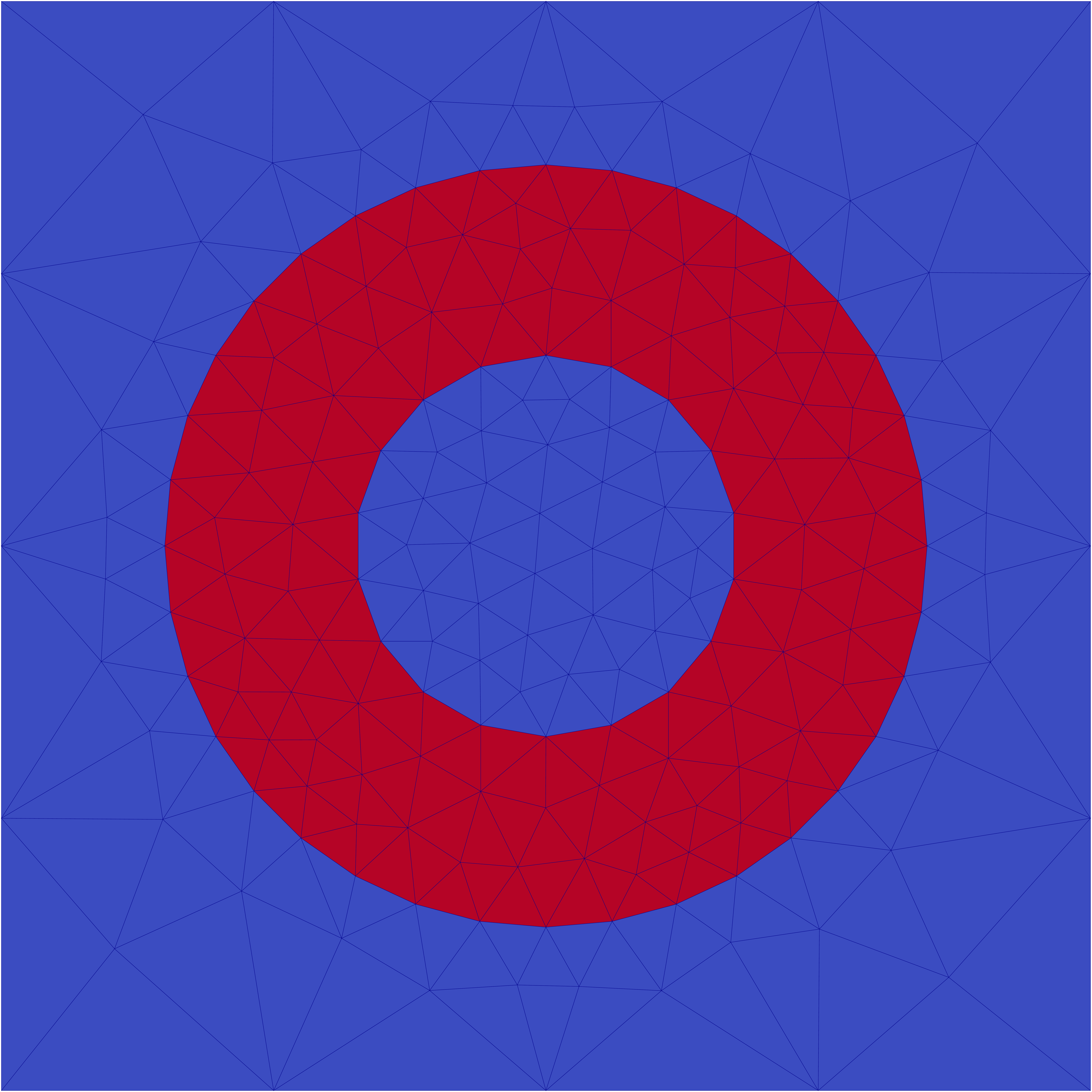}
    \end{minipage}
    \begin{minipage}{.49\linewidth}
        \settowidth{\imagewidth}{\includegraphics{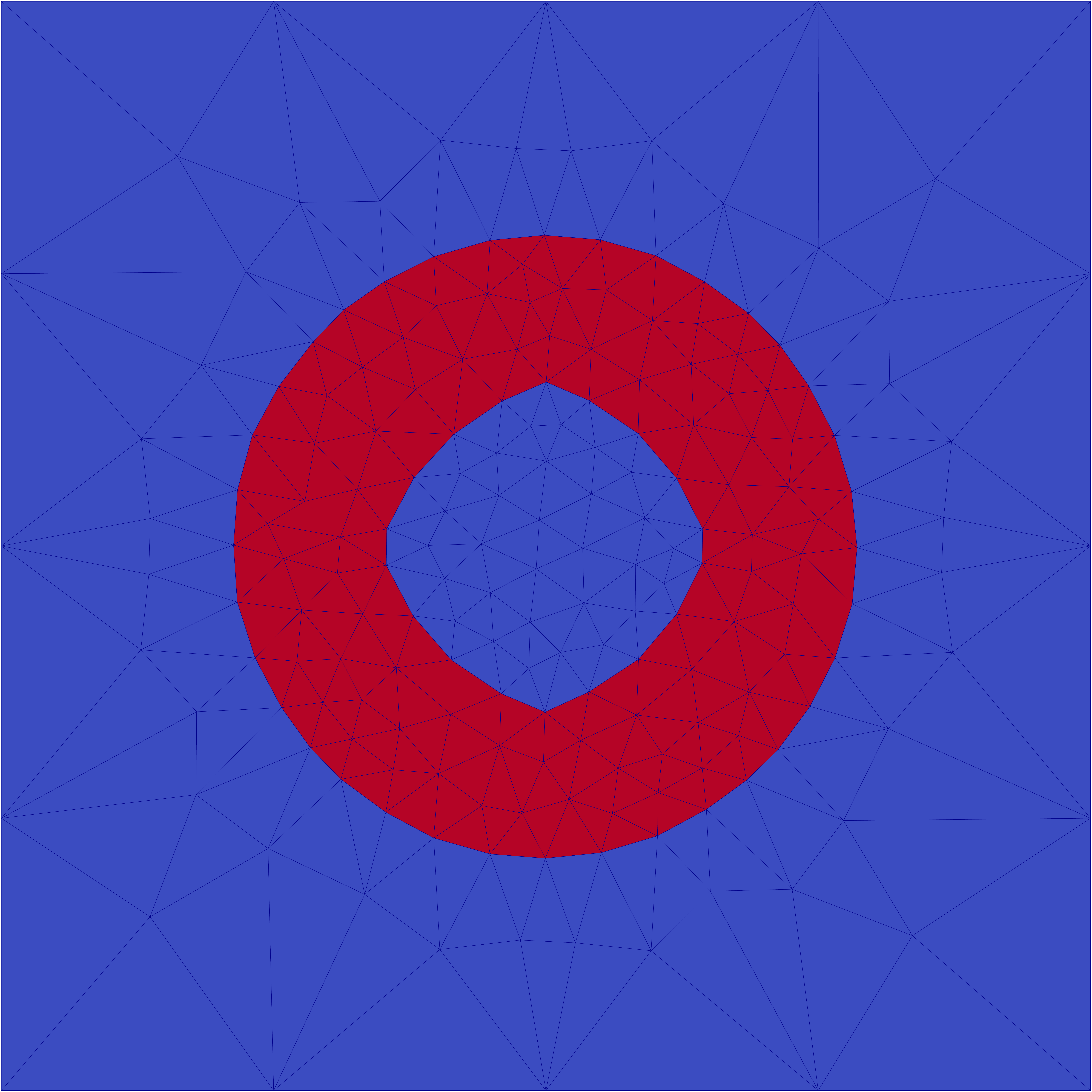}}
        \includegraphics[trim = 0 .5\imagewidth{} .5\imagewidth{} 0, clip,width = .45\linewidth]{experiment_2_refine=0_final_red_blue.pdf}
        \includegraphics[trim = .5\imagewidth{} .5\imagewidth{} 0 0, clip,width = .45\linewidth]{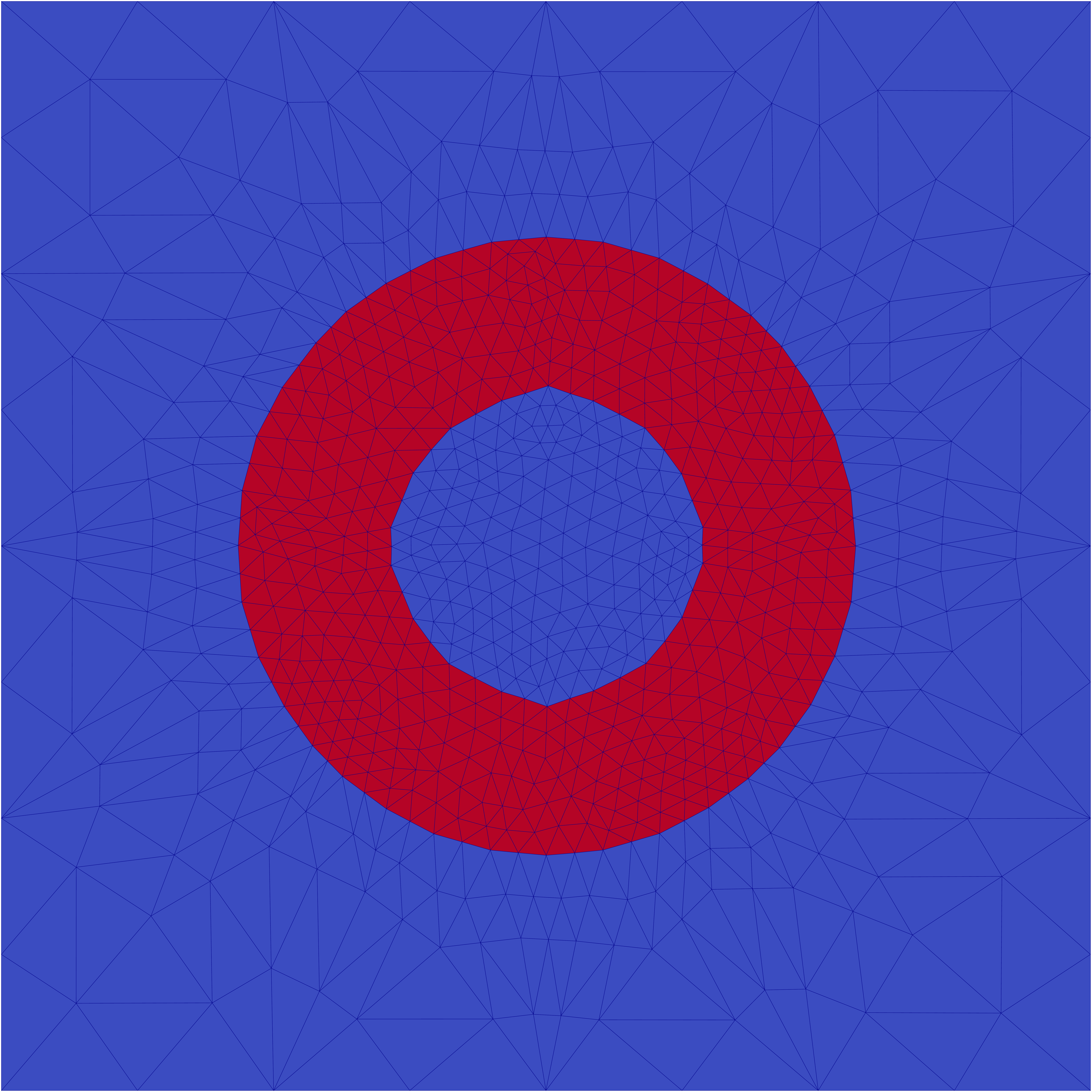}
    
    \noindent
        \includegraphics[trim = 0 0 .5\imagewidth{} .5\imagewidth{}, clip,width = .45\linewidth]{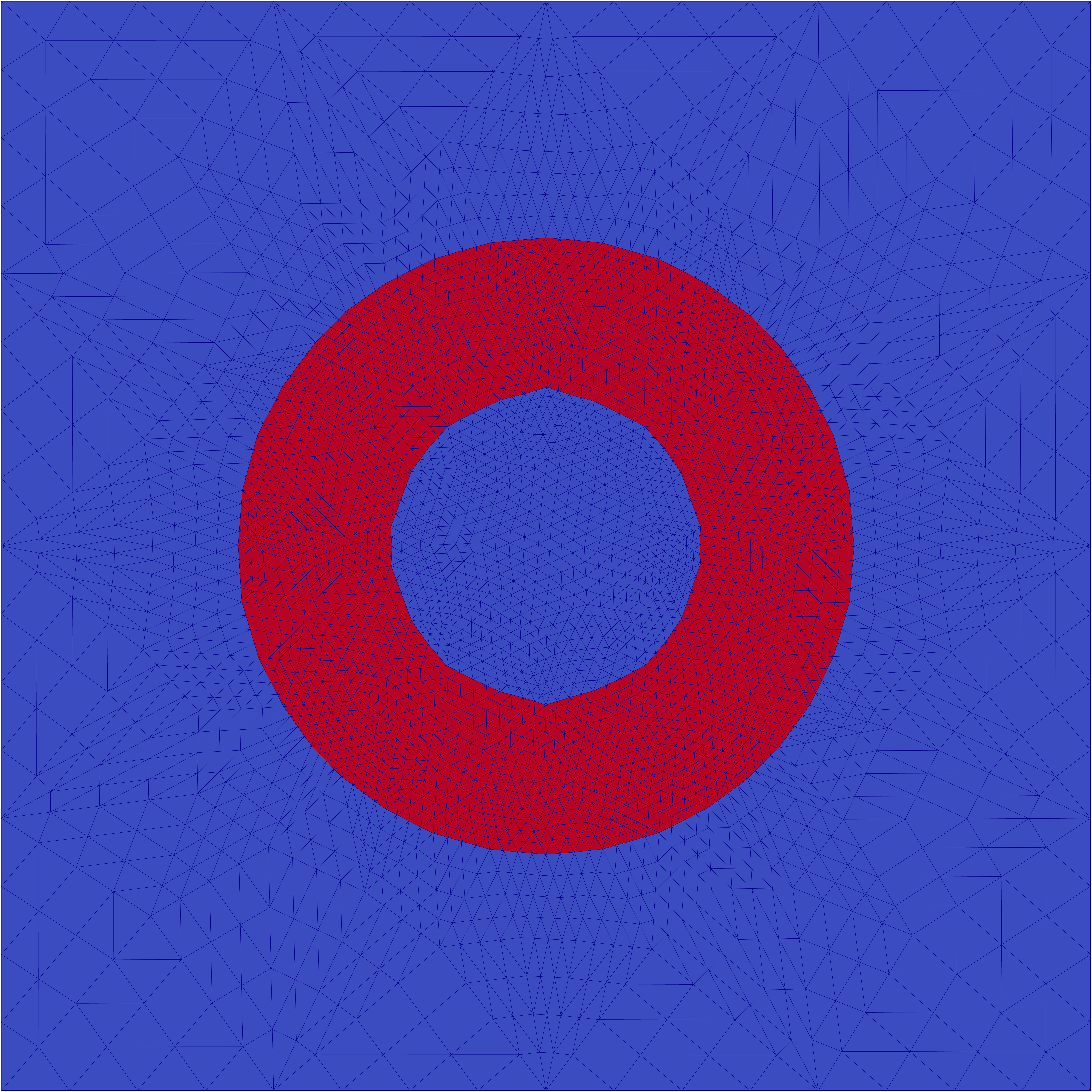}
        \includegraphics[trim = .5\imagewidth{} 0 0 .5\imagewidth{}, clip,width = .45\linewidth]
        {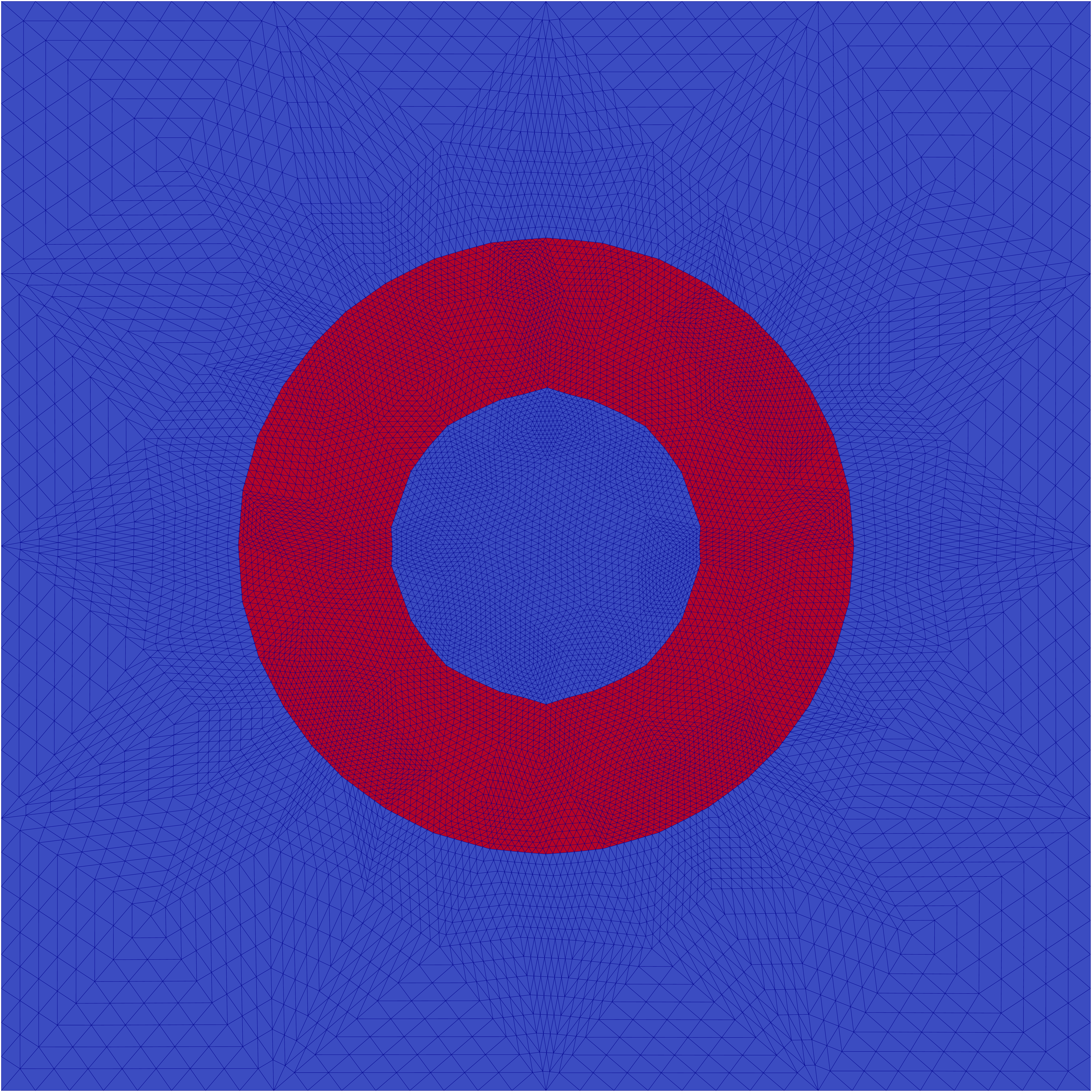}
    \end{minipage}
    \caption{On the left is the initial domain and on the right is a quarter of the final domains for the experiment in Section \ref{sec:exp2}, refinements increasing from top left to bottom right.
    Taking the maximum over all triangles, the radius ratio for the initial grid is $\tilde{\sigma}_0 \approx 1.787287$ and for the final, most fine, grid is $\tilde{\sigma}_f \approx 1.917235$.
    }
    \label{fig:exp2:domains}
\end{figure}

\subsubsection{Experiment 3}\label{sec:exp3}

For this experiment, we choose $f = 1$ and consider $j(x,u,z):= \frac{1}{2}| z + \frac{x}{2}|^2$ along with a penalty for the volume as in Remark \ref{rem:penalty}, where we set $m_0 = 4$.
Without the penalty term and for this choice of $j$, it holds that for any $r>0$, the balls $B_r(0)$ would be a minimiser with $0$ energy.
By penalising the volume to be equal to $4$, one has that the minimiser will be the ball of radius $r = \frac{2}{\sqrt{\pi}}$, with $\mathcal{J}(\Omega^*) = 0$.

As remarked at the beginning of this section, this experiment is performed slightly differently to the previous two.
While we will cascade to a finer mesh after a maximum of 15 steps, we will also continue with the shape optimisation to provide the first shape which is produced with an Armijo step of $2^{-11}$.
This allows for a fair comparison of how the energy and Hausdorff complementary distance appear when the shape is approximately stationary.
When we refine through the cascade, we will increase the penalty parameter so that it scales like $h^{-\frac{1}{2}}$.

{For this experiment, we use a symmetric grid, rather than one generated by \texttt{pygmsh}.}
The initial domain, given by $\Omref = (-1,1)^2$, appears in red on the left of Figure \ref{fig:exp3:domains}, with the hold all in blue.

In Figure \ref{fig:exp3:energy}, we see the energy and the discrete Hausdorff complementary distance \eqref{eq:DHCD} for the experiment along the iterations.
\begin{figure}
    \centering
    \includegraphics[width = .45\linewidth]{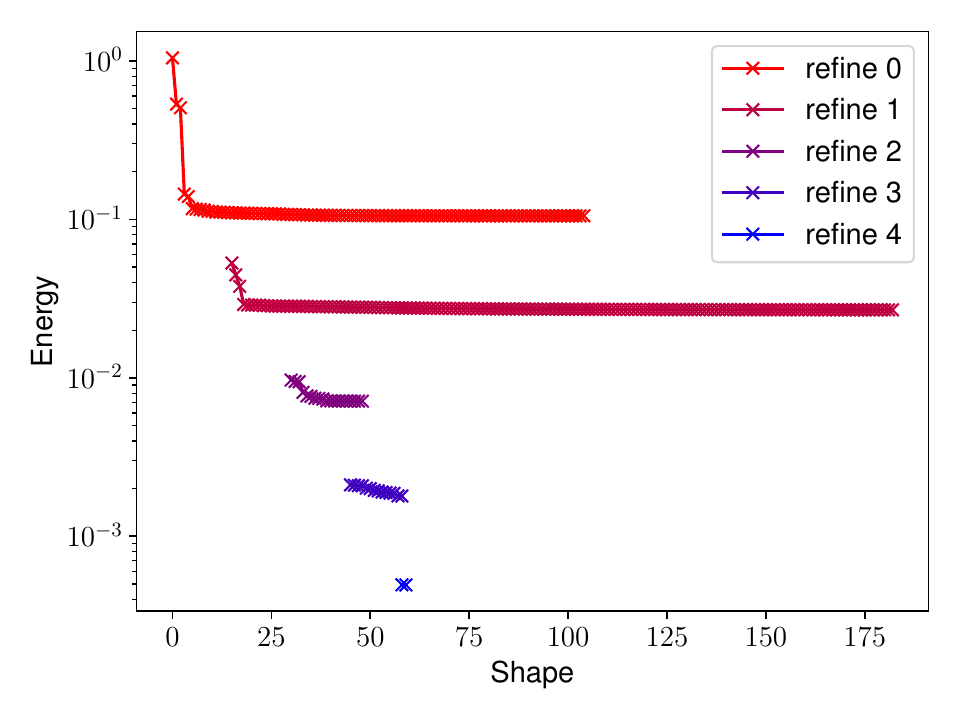}
    \hfill
    \includegraphics[width = .45\linewidth]{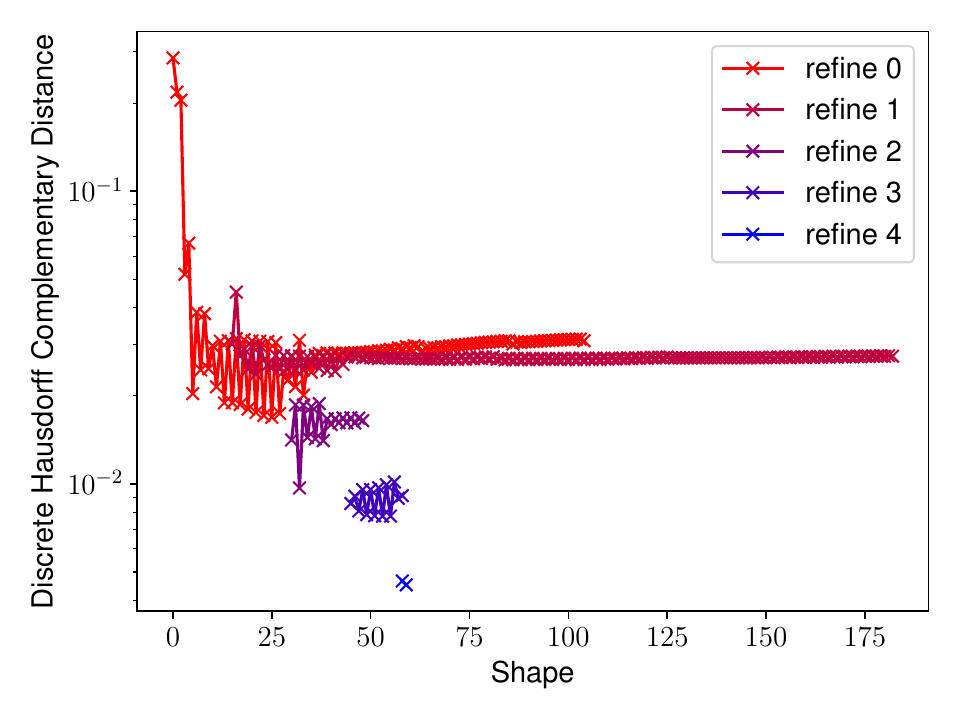}
    \caption{On the left is the energy and on the right discrete Hausdorff complementary distance along the shape iterates for the experiment in Section \ref{sec:exp3}.
            We see that the energy is reducing along the shapes, jumping to a lower energy when the mesh is refined.
            The coarser grids require many more shape updates to reach the prescribed convergence criteria than the finer grids.
            }
    \label{fig:exp3:energy}
\end{figure}
In Table \ref{tab:exp3}, we tabulate the mesh size of the reference domain for each of the approximately converged shape along with the associated energy and discrete Hausdorff complementary distance.
We also provide the \emph{experimental order of convergence}, which for a given functional $E(h)$, here depending on the size $h$ of the reference mesh, is defined by
\[
EOC := \frac{\ln{E(h_1)}-\ln{E(h_2)}}{\ln{h_1}-\ln{h_2}}.
\]
\begin{table}
    \centering
    \begin{tabular}{c|c|c|c|c|c}%
    $h$ &$\mu_h$ &Energy  &EOC Energy  &HCD &EOC HCD
    \\\hline
    0.5 &0.5    &0.105327   &-- &0.0308699  &--
    \\\hline
    0.25    &$\frac{1}{\sqrt{2}}$ &0.0268579  &1.97146    &0.0273133  &0.176597
    \\\hline
    0.125   &1  &0.00712922 &1.91353    &0.016456   &0.730993
    \\\hline
    0.0625  &$\sqrt{2}$   &0.00179752 &1.98774    &0.00912092 &0.851362
    \\ \hline
    0.03125 &2  &0.000493593    &1.86461    &0.00451768 &1.0136
    \end{tabular}
    \caption{Energy and discrete Hausdorff complementary distance for the final domains for each refinement level for the experiment in Section \ref{sec:exp3}.
        We see a clear decrease in the Energy and the Discrete Hausdorff Complementary Distance as we refine the mesh.
    }
    \label{tab:exp3}
\end{table}
The value of $\|D\Phi_h\|_{L^\infty}$ and $\|D\Phi_h^{-1}\|_{L^\infty}$ along the iterations is found in Figure \ref{fig:exp3:DPhi}
\begin{figure}
    \centering
    \includegraphics[width = .45\linewidth]{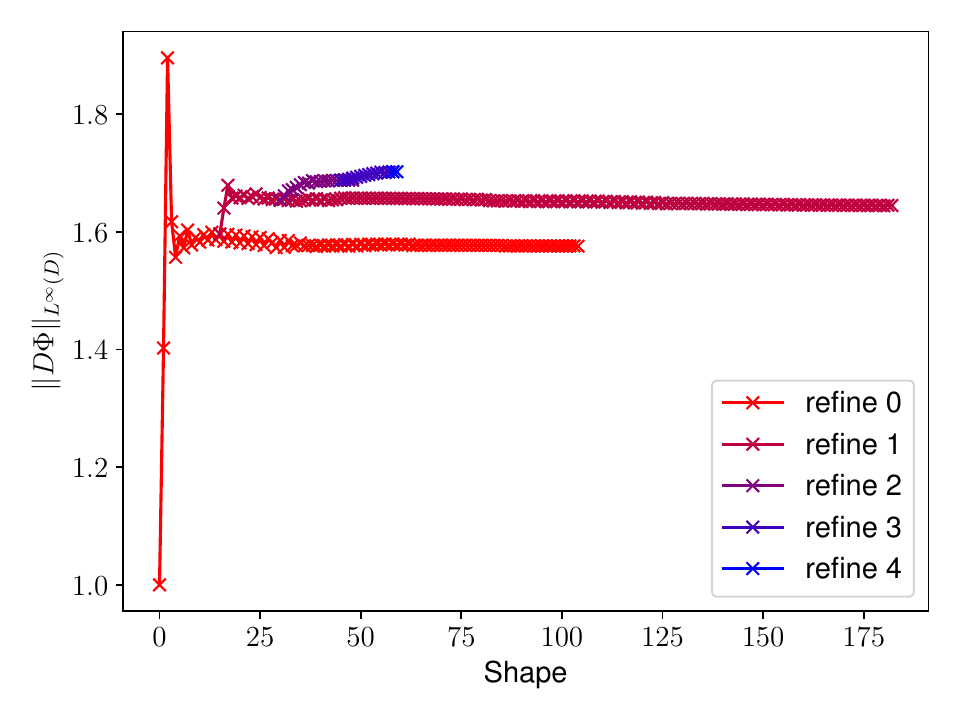}
    \hfill
    \includegraphics[width= .45\linewidth]{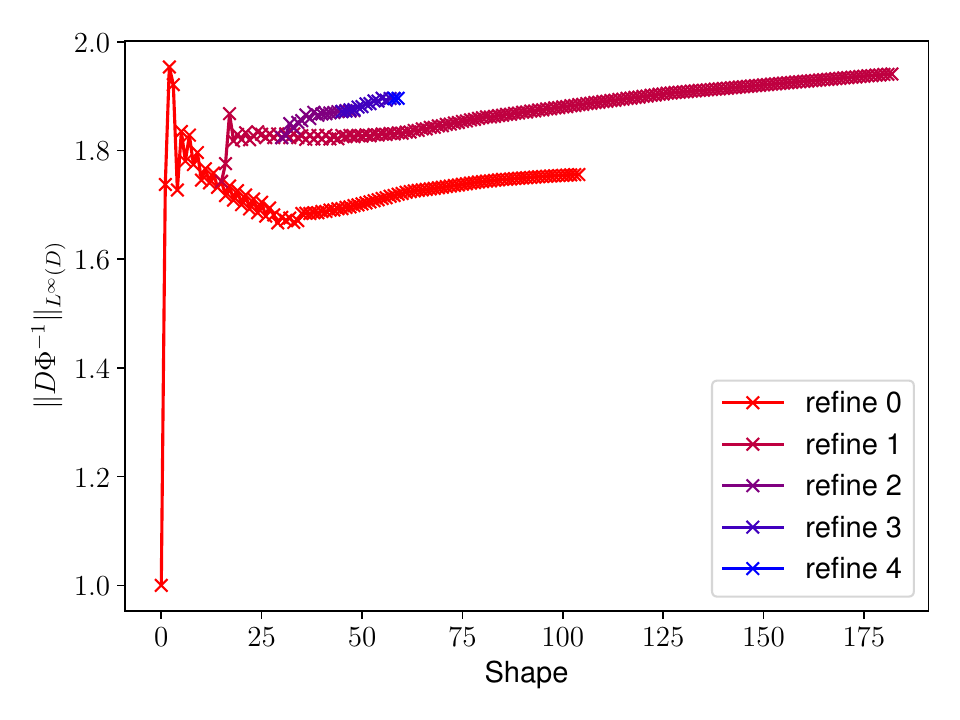}
    \caption{Values for $\|D\Phi_h\|_{L^\infty}$ (left) and $\|D\Phi_h^{-1}\|_{L^\infty}$ (right) along the shape iterates for the experiment in Section \ref{sec:exp3}.
        As in the previous experiments, the values of $\|D\Phi_h\|_{L^\infty}$ and $\|D\Phi_h^{-1}\|_{L^\infty}$ abruptly jump up at the start.
        We see for the coarser grids that $\|D\Phi_h^{-1}\|_{L^\infty}$ increases a large amount relative to the finer grids.
        For the finer grids, both values appear very stable with minimal increase beyond the initial jump.
    }
    \label{fig:exp3:DPhi}
\end{figure}
The final domains are given on the right of Figure \ref{fig:exp3:domains}.
    
\begin{figure}
    \centering
    \begin{minipage}{.49\linewidth}
        \includegraphics[width = .9 \linewidth]{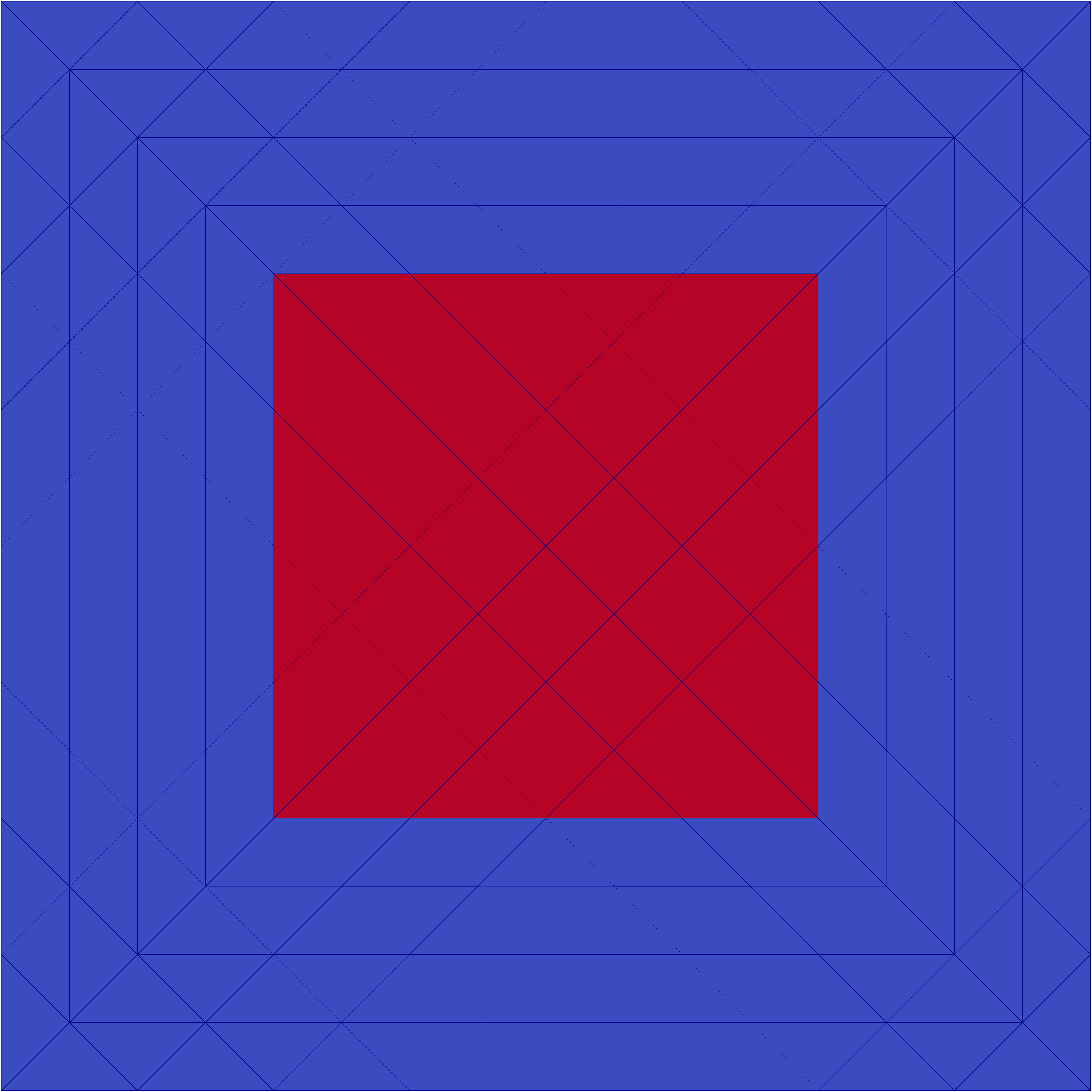}
    \end{minipage}
    \begin{minipage}{.49\linewidth}
        \settowidth{\imagewidth}{\includegraphics{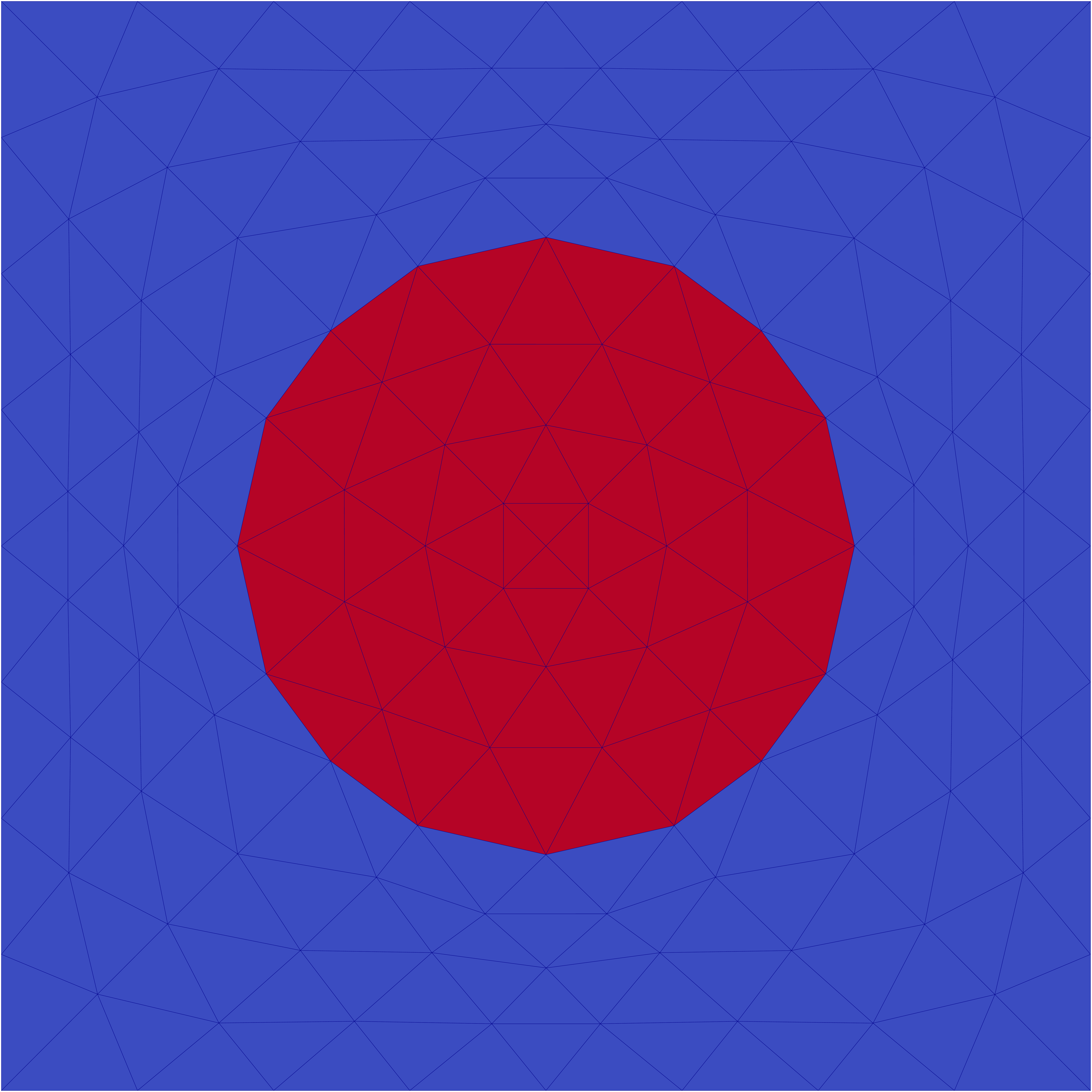}}
        \includegraphics[trim = 0 .5\imagewidth{} .5\imagewidth{} 0, clip,width = .45\linewidth]{experiment_3_refine=0_final_red_blue.pdf}
        \includegraphics[trim = .5\imagewidth{} .5\imagewidth{} 0 0, clip,width = .45\linewidth]{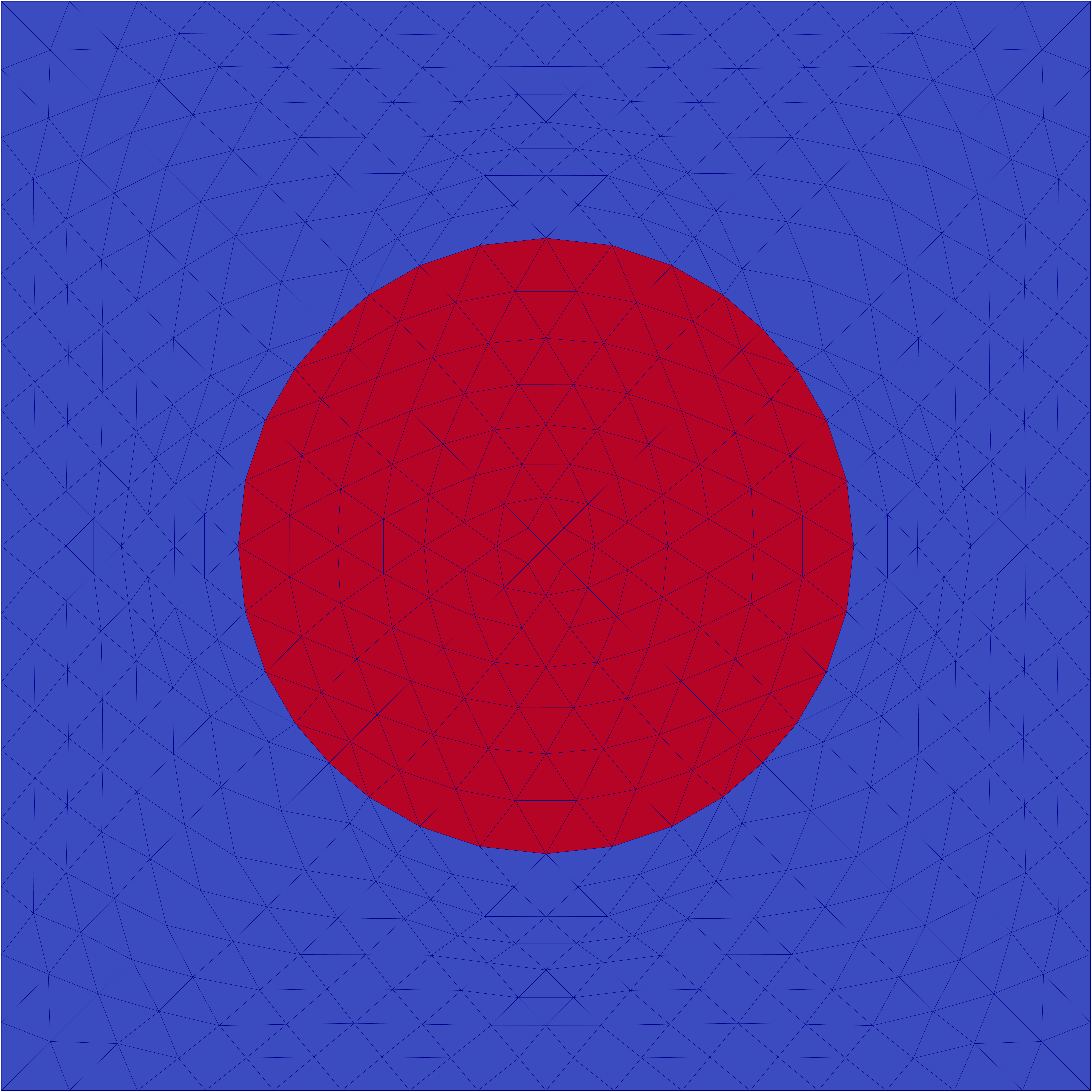}
    
    \noindent
        \includegraphics[trim = 0 0 .5\imagewidth{} .5\imagewidth{}, clip,width = .45\linewidth]{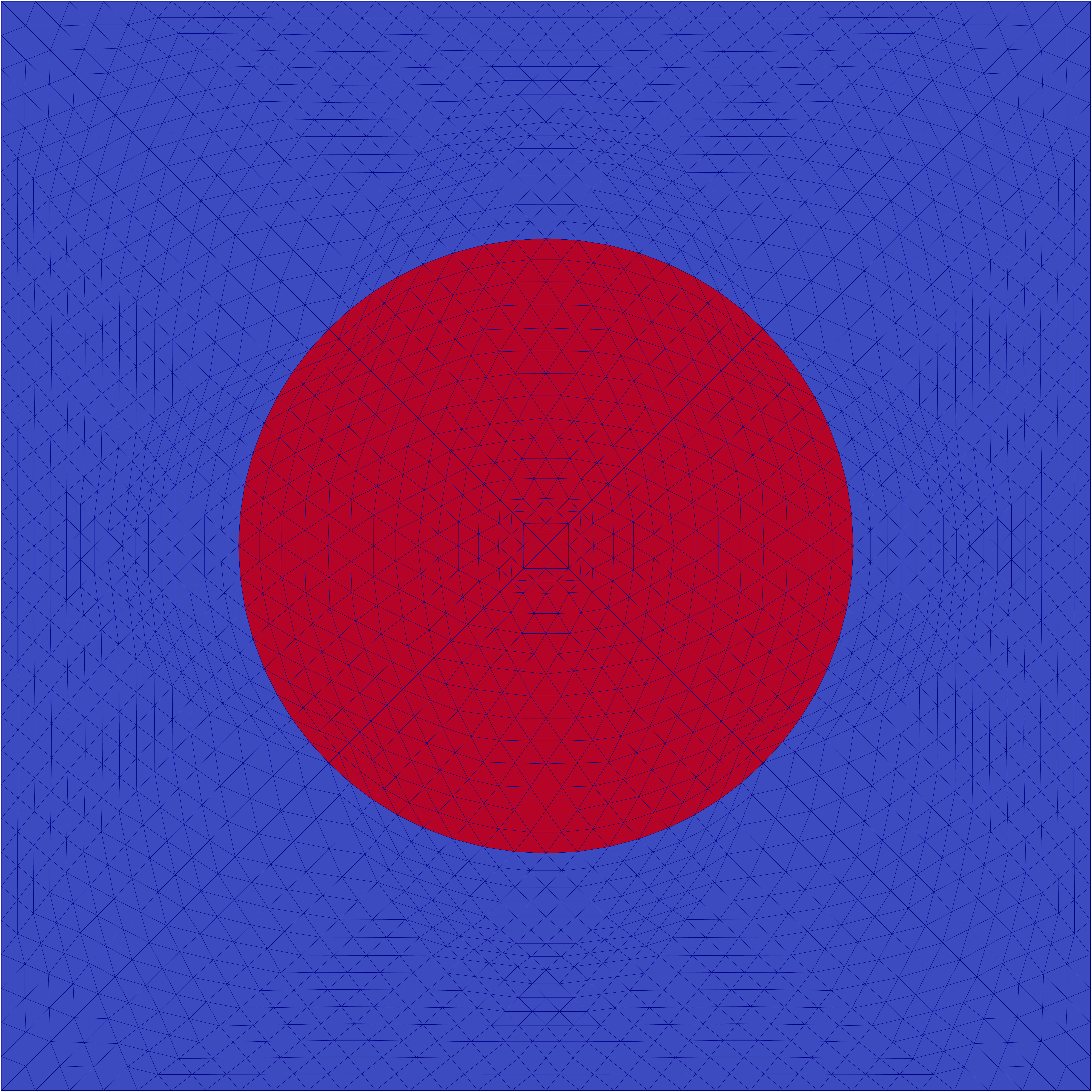}
        \includegraphics[trim = .5\imagewidth{} 0 0 .5\imagewidth{}, clip,width = .45\linewidth]
        {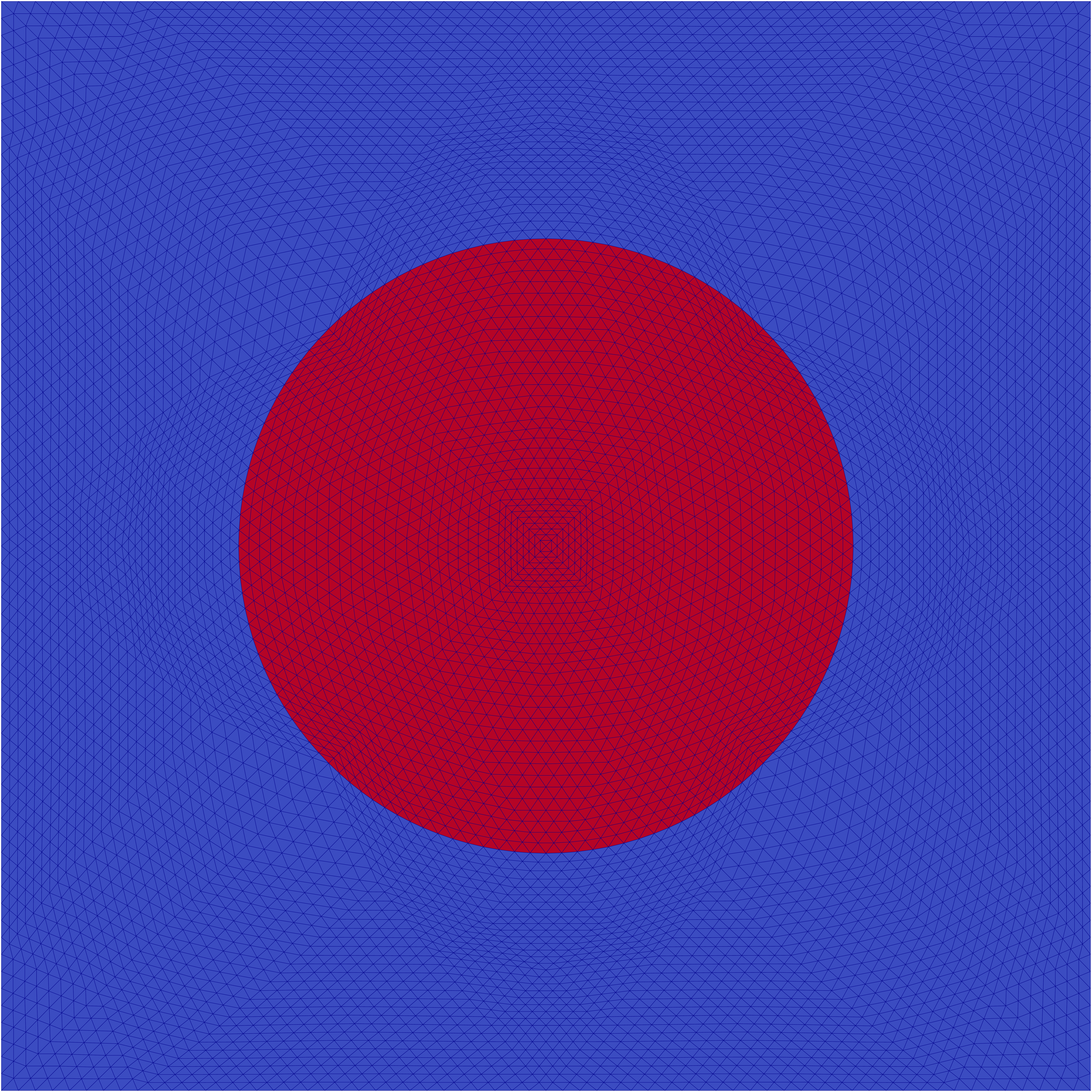}
    \end{minipage}
    \caption{Final domains for the experiment in Section \ref{sec:exp3}, refinements increasing from top left to bottom right.
    The most refined mesh is not shown due to the very small triangles.
    Taking the maximum over all triangles, the radius ratio for the initial grid is $\tilde{\sigma}_0 = \frac{1}{2} + \frac{1}{\sqrt{2}} \approx 1.207107$ and for the final, most fine, grid is $\tilde{\sigma}_f \approx 1.489403$.
    }
    \label{fig:exp3:domains}
\end{figure}

\section{Conclusions}
This work presents a numerical finite element solution framework for discrete PDE constrained shape optimisation in the $W^{1,\infty}$-topology based on the steepest descent method with Armijo step size rule.
In Theorem \ref{conv1}, global convergence of this method is shown for a fixed discretisation parameter.
Moreover, in Theorem \ref{conv3} it is shown that a sequence of discrete stationary shapes under assumption (A2) converges with respect to the Hausdorff complementary metric to a stationary point of the limit problem \eqref{prob0} for the mesh parameter tending to zero.
The proof of this result is based on the continuity of the Dirichlet problem with respect to the Hausdorff complementary metric in terms of $\gamma$--convergence. 

In future work our numerical concept and convergence analysis could be extended to the numerical investigation of shape Newton methods in the $W^{1,\infty}$-topology, which we addressed in \cite{DecHerHin23}, {or that of transformations which preserve some geometric quantity, as in \cite{HerPinSie23}}.

\printbibliography

\end{document}